\documentclass[hidelinks,onefignum,onetabnum]{siamart220329}

\usepackage{amsmath, amssymb}
\usepackage{mathrsfs}
\usepackage{mathtools}
\usepackage{xcolor}
\usepackage{bbm} 
\usepackage{hyperref}
\usepackage{bm}
\usepackage{physics}
\hypersetup{colorlinks = true, linkcolor=red, citecolor=blue}
\usepackage{siunitx}

\usepackage{marginnote}
\definecolor{misson}{rgb}{1,.5,.1}
\definecolor{green}{rgb}{0,.8,.1}

\newcommand{\gd}[1]{\textcolor{blue}{ #1}}

\usepackage[version=4]{mhchem}

\usepackage{subcaption}
\usepackage{caption}
\usepackage{graphicx}
\graphicspath{{code/img}}
\usepackage{diagbox}

\usepackage{algorithm}
\usepackage{algorithmic}

\DeclareMathOperator{\supp}{supp}

\newcommand{\cA}{c_A}

\newcommand{\R}{\mathbb{R}}
\newcommand{\Sobolev}[1]{\mathrm{H}^{#1}} 
\newcommand{\Ldeux}{\mathrm{L}^{2}} 
\newcommand{\Lp}[1]{\mathrm{L}^{#1}}
\newcommand{\tildeprefactor}{\gamma}
\newcommand{\hatprefactor}{\widehat \gamma}
\newcommand{\K}{{K}}
\newcommand{\normA}{A}
\newcommand{\normAk}{A_k}
\newcommand{\normAMp}{A_{\K +1}}
\newcommand{\dualA}{A^{-1}}
\newcommand{\dualAk}{A_k^{-1}}
\newcommand{\dualAMp}{A^{-1}_{\K +1}}
\newcommand{\VoneD}{W}
\newcommand{\eigvec}{\varphi}
\newcommand{\shift}{\varsigma}
\newcommand{\spectrum}{\sigma}
\AtBeginDocument{\RenewCommandCopy\qty\SI}

\newcommand{\N}{\mathbb{N}}
\newcommand{\J}{{J}}

\newtheorem{assumption}[theorem]{Assumption}
\newtheorem{remark}[theorem]{Remark}

\title{A posteriori error estimates for \\ Schr\"odinger operators
discretized with \\ linear combinations of atomic orbitals\thanks{This project has received funding from the European Research Council (ERC) under the European Union’s Horizon 2020 research and innovation programme (grant agreement EMC2 No 810367). GD was supported by the French ‘Investissements d’Avenir’ program, project Agence Nationale de la Recherche (ISITE-BFC) (contract ANR-15-IDEX-0003). GD was also supported by the Ecole des Ponts-ParisTech.}}

\author{Mi-Song Dupuy\thanks{Laboratoire Jacques-Louis Lions, Sorbonne Universit\'e, Paris, France
  (\email{mi-song.dupuy@sorbonne-universite.fr}).}
\and Geneviève Dusson\thanks{Université de Franche Comté, CNRS, LmB (UMR 6623), F-25000 Besançon, France.
(\email{genevieve.duson@univ-fcomte.fr}).}
\and Ioanna-Maria Lygatsika\thanks{
CEA, DAM, DIF, F-91297 Arpajon, France; and Université Paris-Saclay, Laboratoire Matière en Conditions Extrêmes, 91680 Bruyères-le-Châtel, France. (\email{ioanna-maria.lygatsika@cea.fr}).}}
\begin{document}

\maketitle

\begin{abstract} We establish guaranteed and practically computable \textit{a
posteriori} error bounds for source problems and eigenvalue problems
involving linear Schrödinger operators with atom-centered potentials discretized with linear combinations of atomic orbitals. 
We show that the energy norm of the discretization error can be estimated
by the dual energy norm of the residual, that further decomposes into atomic
contributions, characterizing the error localized on atoms. Moreover, we show
that the practical computation of the dual norms of atomic residuals
involves diagonalizing radial Schrödinger 
operators which can easily be precomputed in practice. 
We provide numerical illustrations of the performance of such \textit{a
posteriori} analysis on several test cases, showing that the error bounds accurately estimate the error, and that the localized error components allow for optimized adaptive basis sets. 
\end{abstract}

   

\tableofcontents

\section{Introduction}

Partial differential equations involving Schr\"odinger operators are routinely solved in many scientific fields, such as chemistry, materials science, and condensed matter physics. 
In practice, such equations are discretized and solved on a chosen finite-dimensional space, which introduces an error with respect to the exact solution of the considered equation.
Since accurate approximations of the solutions are often key for practical applications, it is important to guarantee that the error between the exact and the computed solution is below a given tolerance. This can be provided by means of {\it a posteriori} error estimation. 
Moreover, separating the total error bound into local components,
one can build adaptive basis sets, which can be used to reduce the total computational cost leading to a given accuracy. 

This work is motivated by electronic structure calculations 
where Schr\"odinger operators are ubiquitous. In this field, mainly two types of discretization bases coexist. First localized basis sets called linear combinations of atomic orbitals (LCAO) are often used for the simulation of  molecular systems.
They consist in functions centered on the positions of the nuclei in the molecule, and are often Gaussians or products of Gaussians and polynomials. 
Their main advantage is that the required quadratures to compute the discretization operator and overlap matrix elements can be performed analytically, hence very quickly. 
Moreover, these basis functions have been highly optimized over the years, making it possible to obtain good approximations with very few basis functions, of the order of a dozen per atom in the system.
Second periodic Fourier bases called planewaves are widely used, mostly for periodic systems such as materials systems.
We refer to~\cite{Perlt2021-ya} for a monograph on the subject. 
In electronic structure calculations, finite elements are rarely used apart from a few computer codes, namely DFT-FE~\cite{Motamarri2020-ys} for large-scale simulations and HelFEM for diatomic molecules~\cite{Lehtola_code}. 

On the contrary {\it a posteriori} estimators have been widely developed for finite element discretization methods, both for source problems~\cite{Verfurth2013-xi,Destuynder1999-wq,Ern2015-gu} and eigenvalue problems~\cite{Carstensen2014-jr,Hu2014-cp,Liu2015-xk,Cances2017-oc}, see also references therein. 
Estimators for planewave methods in electronic structure calculations have only been proposed recently, see e.g.~\cite{antoine_apost,Cances2020-sw} for linear eigenvalue Schr\"odinger equations, or~\cite{dusson2017,Cances2022-mg,Bordignon2024-fp} for nonlinear eigenvalue problems; see also~\cite{Dusson2023-fv} for an overview on the subject, and references therein.
The {\it a posteriori} error analysis for LCAOs is even scarcer. 
Even the {\it a priori} analysis seems to be only investigated in~\cite{bachmayr2014error,Scholz_Yserentant_2017} for a specific type of LCAOs.
One major difficulty is that the problem is posed over an unbounded domain, which is fundamentally different from a finite element approach where the basis functions have local support.
To our knowledge, the only \textit{a
posteriori} analysis over unbounded domains for
generic discretizations in the context of electronic structure calculations is proposed in~\cite{maday2003error} for the nonlinear eigenvalue problem called Hartree--Fock, but the error bounds are not fully guaranteed. 

Once {\it a posteriori} error bounds are available, it is often interesting to use them to guide adaptive strategies, refining the parameters of the simulations in a way that minimizes the computational cost leading to a prescribed accuracy.
In the context of finite element simulations, it is common to refine at each step a proportion of mesh elements following the D\"orfler 
marking strategy~\cite{Dorfler1996-vy}.
Such adaptive strategies are e.g. proposed for electronic structure calculations in~\cite{Chen2014-qu,Dai2015-ud}. For planewave discretizations adaptive strategies are presented in~\cite{Liu2022-jo,AAMM-16-636}. Other adaptive strategies exist such as two-grids methods~\cite{Xu*2002-xm,Henning2014,Chakir2016}, see also~\cite{Dusson2023-fv} and references therein. 
To the best of our knowledge, such adaptive basis sets have not been proposed so far for LCAOs.

The aim of this work is therefore to propose such missing {\it a posteriori} analysis as well as adaptive strategies for LCAOs in the context of Schr\"odinger source problems and eigenvalue problems.  
As in~\cite{conforming}, the analysis relies on the estimation of a dual norm of the residual, that is shown to be an upper bound of the eigenfunction error in energy norm. In order to efficiently evaluate this dual norm, we construct a partition of unity that spatially
localizes contributions on different atomic regions, similarly to~\cite{chen2016}. 
This is in spirit close to other error localization strategies such as~\cite{blechta2020localization} for the localization of dual norms of bounded linear functionals.
It is also reminiscent of the multiscale finite element (MsFEM) method~\cite{Babuska_Lipton_2011,Ma_Scheichl_Dodwell_2022} where the domain is partitioned into overlapping regions of the whole domain.
For each small region, a carefully chosen local eigenvalue problem is solved. 
The discretization basis of the full problem is composed of the eigenfunctions of each of the local eigenvalue problems. 
With this choice of basis, it is possible to
estimate the error of the approximate solution on the full domain.
The authors in~\cite{Ma_Scheichl_Dodwell_2022} have been able to characterize the local error bounds.
Their analysis unfortunately does not apply to our context, as they work in a bounded domain, with operators that only have discrete spectra.

The contribution of this work is two-fold. First we develop the theory of
residual-based estimators over unbounded domains deriving fully guaranteed and computable error bounds for source problems and eigenvalue problems. Thanks to a partition of unity we separate the total error into local components, opening the way for adaptive bases in this context. 
To our knowledge this is the first time that rigorous {\it a posteriori} error estimates are proposed suited to Schr\"odinger operators discretized with LCAO bases.
Second we present numerical 
results illustrating the good performance of the error estimates, and we use the local bounds in an adaptive strategy
to generate adaptive atomic orbital basis sets. 
We perform simulations in 1D with Hermite polynomials basis sets and Coulomb-type potentials, using the setting previously introduced in~\cite{GaussOpt}, as well as in 3D on the \ce{H2+} and \ce{LiH^3+} molecules, using the HelFEM finite elements code~\cite{Lehtola_2019_review, Lehtola_code} for the reference calculations.

This article is organized as follows. 
In Section~\ref{sec:setting} we describe the setting of the {\it a posteriori} bounds for the source problem and the eigenvalue problem. 
Section~\ref{sec:results} contains the theoretical results of this article, namely the {\it a posteriori} error bounds for the two considered problems.
In Section~\ref{sec:num_results_adapt} we present numerical results.
The proofs of the error bounds are gathered in Section~\ref{sec:proofs}.
We conclude in Section~\ref{sec:concl}.

\section{Theory}
\label{sec:setting}

In this section, we present the mathematical setting, the source problem and eigenvalue problem of interest. We then present the LCAO variational approximation and a few preliminary definitions that will be useful in the rest of the article.

\subsection{Setting}

Let $\Ldeux(\mathbb{R}^d)$, for $d=1,2,3$, 
be the standard Lebesgue space of square integrable functions on $\mathbb R^d$ endowed with its scalar product 
\[
    \forall \, v,w \in \Ldeux(\mathbb R^d),\ \langle v,w\rangle_{\mathbb R^d} = \int_{\mathbb R^d}v(x)w(x)\dd{x}.
\]
The associated norm $\|\cdot\|_{\mathbb R^d}$ is given by
\(
    \|v\|_{\mathbb R^d} = \langle v,v\rangle_{\mathbb R^d}^{1/2}.
\)
For $\Omega \subset \mathbb{R}^d$, we denote by 
\[
\forall \, v,w \in \Ldeux(\Omega),\ \langle v,w\rangle_{\Omega} = \int_{\Omega}v(x)w(x)\dd{x},
\]
the scalar product on $\Ldeux(\Omega)$. The corresponding norm is denoted by $\|\cdot\|_{\Omega}$.

We consider a potential $V$ satisfying the following assumption.
\begin{assumption}
\label{assump:potential}
    Let $\VoneD:\mathbb R^+\rightarrow\mathbb R$ be such that the potential defined by
\[
    V : \left\lbrace \begin{aligned} \mathbb{R}^d & \to \mathbb{R} \\ x & \mapsto \VoneD(|x|) \end{aligned} \right. 
\]
is in $\Lp{p}(\mathbb{R}^d) + \mathrm{L}^{\infty}_\varepsilon(\mathbb{R}^d)$\footnote{Note that the notation 
$\mathrm{L}^{\infty}_\varepsilon(\mathbb{R}^d)$ is borrowed from~\cite[Definition 1.8]{Lewin_2022} and does not depend on $\varepsilon$.},
with $p = 1$ for $d=1$, and $p=3/2$ for $d=2,3$,
\emph{i.e.} for any $\varepsilon>0$, $V$ can be written as $V = V_p+V_\infty$ with $V_p \in \Lp{p}(\mathbb R^d)$, $V_\infty \in \mathrm{L}^\infty(\mathbb R^d)$ with $\|V_\infty\|_{\mathrm{L}^{\infty}} \leq \varepsilon$.
\end{assumption}

A typical potential of interest that is within our setting is the Coulomb potential $V(x) = \frac{1}{|x|}$ in dimension $d=3$.
Under \gd{\cref{assump:potential}}, the operator $-\frac{1}{2} \Delta + V$ is a lower-bounded self-adjoint operator acting on $\Ldeux(\mathbb R^d)$ with domain $\Sobolev{2}(\mathbb R^d)$, where $\Sobolev{2}(\mathbb R^d)$ is the standard Sobolev space of order 2 on $\mathbb R^d$. Its essential spectrum is given by $\spectrum_{\mathrm{ess}}(-\frac{1}{2} \Delta + V) = [0,\infty)$~\cite[Corollaire 5.35]{Lewin_2022}.
Additionally, we assume that $-\frac{1}{2} \Delta+V$ has eigenvalues below its essential spectrum. 

\begin{assumption}
\label{assump:negative_eigvals}
    The operator $-\frac12\Delta+V$ as an operator acting on $\Ldeux(\mathbb{R}^d)$ with domain $\Sobolev{2}(\mathbb{R}^d)$ has negative eigenvalues $(\varepsilon_k)_{k \in I}$, where $I$ is not empty and countable.
\end{assumption}

Given parameters $R_1,\ldots,R_\K \in \mathbb R^d$ corresponding to atomic positions, for $k=1,\ldots,\K$,  let $V_k(x) = \VoneD(|x-R_k|)$ and
consider the linear Hamiltonian operator of Schrödinger-type defined by
\[
  A = -\frac12\Delta + \sum_{k=1}^\K V_k + \shift,
\]
where $\shift \in\mathbb R$ is a shift factor satisfying the following assumption.
\begin{assumption}
   \label{as:posdefA}
    The shift $\shift\in \R$ is
    such that the operator $A$ is coercive.
\end{assumption}
By Assumption~\ref{assump:potential}, $A$ is self-adjoint on $\Ldeux(\mathbb{R}^d)$ with domain $\Sobolev{2}(\mathbb R^d)$ and its essential spectrum is given by $\spectrum_{\mathrm{ess}}(A) = [\shift,\infty)$.
From Assumption~\ref{assump:negative_eigvals}, we have that $A$ also has eigenvalues below its essential spectrum.
Since $A$ is a coercive self-adjoint operator, the corresponding quadratic form defined as $q_A(v) = \langle v, Av \rangle_{\mathbb R^d}$ for $v \in \Sobolev{2}(\mathbb R^d)$ admits a unique self-adjoint Friedrichs extension denoted by $\widetilde{q_A}$ whose form domain $\mathrm{Q}_A$ satisfies $\Sobolev{2}(\mathbb R^d) \subset \mathrm{Q}_A \subset \Ldeux (\mathbb R^d)$ \cite[Theorem X.23]{Reed_Simon_1975}, and in this case
$\mathrm{Q}_A = \Sobolev{1}(\mathbb R^d)$ \cite[Corollaire 3.19]{Lewin_2022}.
This justifies the definition of the energy norm for $A$ given by  
\[
   \forall \, v \in \Sobolev{1}(\mathbb{R}^d), \quad\|v\|_{\normA} = \sqrt{\widetilde{q_A}(v)}, 
\]
that we will write with a slight abuse of notation, noting that $\widetilde{q_A}(v) = q_A(v)$ for $v \in \Sobolev{2}(\mathbb R^d)$
\[
    \forall \, v \in \Sobolev{1}(\mathbb{R}^d), \quad\|v\|_{\normA} = \langle v,Av\rangle_{\mathbb R^d}^{1/2}. 
\]
Moreover we denote by $\cA$ a positive constant such that
\begin{equation}
\label{eq:cA}
     \forall \,v \in\Sobolev{1}(\mathbb{R}^d), \quad \cA \|v\|_{\normA} \ge \|v\|_{\mathbb R^d}. 
\end{equation}
As $A$ is coercive and self-adjoint on $\Sobolev{2}(\mathbb R^d)$, one can define, 
by spectral calculus,
the symmetric operator $A^{1/2}$ on $\Sobolev{1}(\mathbb R^d)$ 
and $\|v\|_A = \|A^{1/2}v\|_{\mathbb R^d}$ for all $v \in \Sobolev{1}(\mathbb R^d)$ \cite[Théorème 4.24]{Lewin_2022}. 
The associated dual norm is defined by
\begin{align*}
    \forall\, v \in \Ldeux(\mathbb{R}^d), \quad\|v\|_{\dualA} &= \langle v, A^{-1} v \rangle_{\mathbb R^d}^{1/2}.
\end{align*}
We can also define the symmetric square root of $A^{-1}$ denoted by $A^{-1/2}$ and we have that $\|A^{-1/2}v\|^2_{\mathbb R^d} = \langle A^{-1/2}v, A^{-1/2}v \rangle_{\mathbb R^d} = \|v\|_{\dualA}^2$ for any $v \in \Ldeux(\mathbb R^d)$.
Moreover, there holds $A^{-1/2}A^{1/2}v = v$ for any $v \in \Sobolev{1}(\mathbb R^d)$.
Note that the dual norm $\|\cdot \|_{\dualA}$ can be equivalently defined for all $v \in \Ldeux(\mathbb R^d)$ by 
    \begin{equation}
        \label{eq:dualA_equivalent_def}
        \|v\|_{\dualA} = \sup_{w\in \Sobolev{1}(\mathbb R^d)} 
       \frac{\langle v,w\rangle_{\mathbb
       R^d}}{\|w\|_{\normA}}.
    \end{equation}

\subsection{Problem presentation}

In this work, we consider the two following problems on unbounded domains:
\begin{enumerate}
    \item the source problem: for a fixed $f \in \Ldeux(\mathbb{R}^d)$,
\begin{equation}\label{eq:pb_rhs}
    \text{Find }u\in\Sobolev{2}(\mathbb R^d)\text{ solution to }
    Au = f;
\end{equation}
\item the eigenvalue problem
\begin{equation}\label{eq:pb_eig}
    \text{Find }(\lambda_i,\eigvec_i)\in\mathbb R\times \Sobolev{2}(\mathbb R^d)
    \text{ solution to }
        A\eigvec_i = \lambda_i \eigvec_i, \text{ with } \|\eigvec_i\|_{\mathbb R^d} = 1.
\end{equation} 
We denote by $\spectrum_d(A)$ the set of possible eigenvalues $\lambda_i$ solutions to~\cref{eq:pb_eig}, that we order increasingly counting multiplicities.
We also denote by $\spectrum(A) := \spectrum_d(A) \cup \spectrum_{\rm ess}(A)$.
\end{enumerate}

\subsection{Variational approximation}\label{section:variational}

In practice, problems~\cref{eq:pb_rhs} and~\cref{eq:pb_eig} are solved using a 
Galerkin method with the 
atomic orbital (AO) basis set $\{\chi_\mu\}_{1\leq\mu\leq N}$ of size $N$, 
composed of functions centered at atomic positions
\[
    \{\chi_\mu\} = 
    \left\{\xi_{1,1}(x - R_1),\ldots,\xi_{1,n_1}(x - R_1);\ldots
    ;\xi_{\K,1}(x - R_\K),\ldots,\xi_{\K,n_\K}(x - R_\K)\right\},
\]
with $\xi_{i,j}\in \Sobolev{2}(\mathbb R^d)$ called atomic orbital, $n_1,\ldots,n_K\in \N$ such that $\sum_{k=1}^\K n_k = N$.
Typically, in dimension $d=3$, $\xi_{i,j}(x) = \rho(|x|) Y_{\ell m}(\frac{x}{|x|})$, where $\rho$ is a function with exponential decay and $Y_{\ell m}$ is a spherical harmonics (see~\cite[Chapter 6]{Helgaker_Jorgensen_Olsen_2014} for further details).
Defining the Hamiltonian matrix
\(
    A^{\chi} = \left(\langle \chi_\mu,
    A\chi_\nu\rangle_{\mathbb R^d}\right)_{1\leq\mu,\nu\leq N},
\)
and the overlap matrix
\(
    S^{\chi} =
    \left(\langle\chi_\mu,\chi_\nu\rangle_{\mathbb R^d}
    \right)_{1\leq\mu,\nu\leq N},
\)
the discretization of problem~\cref{eq:pb_rhs} in the AO basis set $\chi$ 
then reads as the linear system
\begin{equation}
\label{eq:discrete_evp}
    A^{\chi} a^{\chi} = S^{\chi}f^{\chi},
\end{equation}
with unknown $a^{\chi}\in\mathbb R^N$, where 
$f^{\chi}=(f^{\chi}_\mu)_{1\leq\mu\leq N}$ are the orthogonal projection coefficients of
the source term $f$ on the AO basis set $\chi$ for the natural $\Ldeux$ inner product.
The Galerkin approximation $u_N\in \Sobolev{2}(\mathbb{R}^d)$ of $u$ in the AO basis set $\chi$ can then be
recovered as the LCAO
\[
    \forall x\in\mathbb R^d, \quad u_N(x) = \sum_{\mu=1}^N a^\chi_\mu
    \chi_\mu(x).
\]

As for the discretization of the eigenvalue problem~\cref{eq:pb_eig} in the AO
basis set $\chi$, it reads as the generalized eigenvalue problem: find
$(\lambda_i^\chi,c_i^\chi)\in\mathbb R\times\mathbb R^N$, $i=1,\ldots,N$
such that
\begin{equation}\label{eq:pb_eig_Gal}
    A^\chi c_i^\chi = \lambda_i^\chi  S^\chi c_i^\chi, \quad i=1,\ldots,N.
\end{equation}
Similarly, eigenfunctions can then be recovered for $i=1,\ldots,N$ as
\[
    \forall x\in\mathbb R^d,\quad \eigvec_{iN}(x) = \sum_{\mu=1}^N
    c_{i\mu}^\chi\chi_\mu(x).
\]
For simplicity, we will denote by $\lambda_{iN}$ the eigenvalues in~\cref{eq:pb_eig_Gal} ranked in increasing order counting multiplicities.

\subsection{Residuals}

As in~\cite{Cances2020-sw}, the proposed {\it a posteriori} error estimation relies on the estimation of the dual norm of residuals for the considered equations. We therefore start by defining the residuals for the source and eigenvalue problems as follows: 
\begin{itemize}
    \item for the source problem \cref{eq:pb_rhs}
        \begin{equation*}
            \label{eq:Res_N}   
           \forall u_N \in \Sobolev{2}(\mathbb R^d), \quad  \Res(u_N) = f - Au_N,
        \end{equation*}
    \item for the eigenvalue problem \cref{eq:pb_eig}   
        \begin{equation*}
            \forall (\lambda_{iN},\eigvec_{iN}) \in \R\times \Sobolev{2}(\mathbb R^d), \quad \Res(\lambda_{iN},\eigvec_{iN}) = \lambda_{iN} \eigvec_{iN} - A\eigvec_{iN}.
        \end{equation*}
\end{itemize}

\subsection{Partition of unity}

In general a significant difficulty to obtain guaranteed error bounds is the estimation and/or computation of the dual norm of the residuals, which amounts to solve a source problem over the whole space.
In order to make the estimation of these dual norms possible, the main idea in this work is to take advantage of the radial symmetry of the potentials $V_k$.
To do so, a first idea would be to separate the operator $A$ into $K$ components
\begin{equation}
\label{eq:atomic_decomp}
       A = -\frac{1}{2}\Delta + \sum_{k=1}^\K V_k + \shift = \sum_{k=1}^\K
    \left[ -\frac{1}{2\K} \Delta + V_k + \frac{1}{\K} \shift \right],
\end{equation}
where the operators $-\frac{1}{2\K} \Delta + V_k + \frac{1}{\K} \shift$ are spherically symmetric with respect to $R_k$.
Thus using that $ -\frac{1}{2\K} \Delta + V_k + \frac{1}{\K} \shift$ as a quadratic form defines a norm equivalent to $\Sobolev{1}$, it is possible to use its eigenvalue decomposition, which is simpler due to the spherical symmetry of $V_k$ to estimate the dual norms of the residual.
Unfortunately this approach does not seem viable, 
as already in the case of the Coulomb potential, 
the solution $v_k$ to $-\frac{1}{2\K} \Delta v_k + V_k v_k+ \frac{1}{\K} \shift v_k = f$ for a sufficienty regular $f$ displays a cusp at $R_k$~\cite{Kato_1957} that is not consistent with the cusp at $R_k$ of $u$ solution to $Au = f$. 

We thus proceed slightly differently. 
We introduce a cover $(\Omega_k)_{1 \leq k \leq \K+1}$ of $\mathbb{R}^d$, which is associated to the spherical potentials $(V_k)_{1 \leq k \leq \K}$
with the property
\[
    \bigcup_{k=1}^{\K+1} \Omega_k = \mathbb R^d,\quad\Omega_k =\overline{B}(R_k,r_k), \quad k=1,\ldots,\K,
\]
where $r_k>0$ and $\overline{B}(R_k,r_k)$ is the closed ball centered at $R_k$ of radius $r_k$. Neighboring
subdomains may intersect, i.e. $\Omega_k\cap \Omega_{k'}\neq \emptyset$ when
$k\neq k'$. 

\begin{figure}[t]
    \begin{center}
        \includegraphics[width=0.35\linewidth]{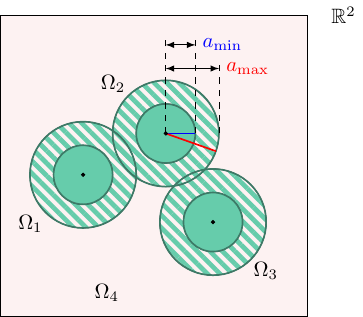}
        \caption{The case of three nuclei in $\mathbb R^2$. Finite cover composed of three balls $\Omega_{1,2,3}$, of radius $r=a_{\max}$ centered at nuclei, and the unbounded set $\Omega_4$ (in light red). The hashed region is the overlap between sets of the cover. For $k=1,2,3$, the partition functions $p_k$ take the value one 
        on the smaller balls of radius $r=a_{\min}$ (rigid green). On $\mathbb{R}^2\setminus(\Omega_1\cup \Omega_2\cup \Omega_3)$, $p_4$ is equal to 1. Otherwise, in the hashed region, $p_k$ is strictly between 0 and 1.}
        \label{fig:omega}
    \end{center}
\end{figure}

We then consider a nonnegative partition of unity subordinate to the finite
cover $(\Omega_k)_{1 \leq k \leq \K+1}$, denoted by $(p_k)_{1\leq k\leq \K+1}$ such that for all $k=1,\ldots,\K+1,$
\begin{equation}\label{eq:poiprop}
    p_k\in \mathrm{C}^\infty(\mathbb R^d), 
    \; \supp(p_k) \subset \Omega_k,
    \; 0 \leq p_k\leq 1,\; \text{and} \;  
    \forall x\in\mathbb R^d, \ \sum_{k=1}^{\K+1} p_k(x) = 1,
\end{equation}
where $\supp$ denotes the support.
An example of such a partition of unity is illustrated in~\cref{fig:omega}, see also Section~\ref{sec:num_results_adapt}.
For the splitting of the {\it a posteriori} estimate, we will need an additional assumption on the partition functions.
\begin{assumption}
\label{assump:partition}
    For all $k=1,\dots,\K+1$, and all $v \in \Sobolev{1}(\mathbb R^d)$, $\sqrt{p_k} v \in \Sobolev{1}(\mathbb R^d)$.    
\end{assumption}
For $k=1,\ldots,\K$, let $A_k$ be the radially symmetric operator
acting on $\Ldeux(\Omega_k)$ with domain $\Sobolev{2}(\Omega_k) \cap \Sobolev{1}_0(\Omega_k)$ defined by
\[
    A_k = -\frac{1}{2} \Delta + V_k + \shift_k,
\]
and let $A_{\K+1}$ be the operator acting on $\Ldeux(\mathbb{R}^d)$ with domain $\Sobolev{2}(\mathbb{R}^d)$ defined by
\[
A_{\K+1} = -\frac{1}{2} \Delta + \shift_{\K+1},
\]
where for $k=1,\ldots, K+1,$ $\shift_k\in\mathbb R$ is a shift factor satisfying
\begin{assumption}
   \label{as:posdef}
    The shifts $\shift_1,\ldots,\shift_{\K+1}$ are 
    such that the operators \linebreak $A_1,\dots,A_{\K+1}$ are coercive.
\end{assumption}
\begin{remark}
    The shifts $\shift_1,\ldots,\shift_{K+1}$ related to the operators $A_1,\ldots,A_{K+1}$ are not {\it a priori related to the shift $\shift$ of the operator $A$, and can be independently chosen.}
\end{remark}

For $1\leq k \leq \K$, the operators $A_k$ acting on $\Ldeux(\Omega_k)$ have a Friedrichs extension with form domain $\Sobolev{1}_0(\Omega_k)$.
We can thus define their corresponding energy norm by
\[
    \forall \, v \in \Sobolev{1}_0(\Omega_k), \quad \|v\|_{\normAk} = \langle v,A_kv\rangle_{\Omega_k}^{1/2},
\]
again with an abuse of notation.
The associated dual norms are given by
\begin{equation*}
    \forall \, v \in \Ldeux(\Omega_k),  \quad\|v\|_{\dualAk} = \langle v,A^{-1}_k v \rangle_{\Omega_k}^{1/2} = \|A_k^{-1/2} v\|_{\Omega_k},
    \quad k=1,\ldots,\K.
\end{equation*}
Likewise, for $A_{\K+1}$, we define the energy and the dual norms associated to $A_{\K+1}$ as
\begin{equation*}
    \begin{aligned}
        \forall \, v \in \Sobolev{1}(\mathbb{R}^d), \quad\|v\|_{A_{\K+1}(\mathbb R^d)} &= \langle v,A_{\K+1} v\rangle_{\mathbb R^d}^{1/2} = \|A_{\K +1}^{1/2}v\|, \\
        \forall\, v \in \Ldeux(\mathbb{R}^d), \quad \|v\|_{\dualAMp} &= \langle v, A^{-1}_{\K+1} v \rangle_{\mathbb R^d}^{1/2} = \|A_{\K+1}^{-1/2}v\|,
    \end{aligned}
\end{equation*} 
where $\langle v,A_{\K+1} v\rangle_{\mathbb R^d}$ is again an abuse of notation for the quadratic form associated to  the self-adjoint Friedrichs extension of $A_{\K+1}$.

\section{Computable \textit{a posteriori} error bounds}
\label{sec:results}

We are now ready to state our main results, namely guaranteed {\it a posteriori} error estimates for the source problem~\cref{eq:pb_rhs} and the eigenvalue problem~\cref{eq:pb_eig}.
In this section we consider the discretization parameter $N$ to be fixed.

\subsection{Guaranteed \textit{a posteriori} error estimates}\label{sec:aee}

In this section, we provide error estimates for the energy norm of the error that are guaranteed but that are not yet fully computable. They indeed involve dual norms of localized residuals.
Let us define the positive constant $C$ by
   \begin{equation}
      \label{eq:C}
      C := 
      1+ 
      \cA^2 \sup_{\mathbb R^d}  
          \left( \sum_{k=1}^{\K+1} -\frac14 \Delta p_k + \frac{|\nabla p_k|^2}
          {8 p_k} + V_k (p_k-1) +  (\shift_k - \shift) p_k \right)^+,
   \end{equation}
   where $\cA$ satisfies~\cref{eq:cA}, and for $a \in \R$, $a^+:=\max
   \{a,0\}$. 
   We state below the error estimation in energy norm for the source problem.
\begin{theorem}[$A$-error estimation for source problem]
   \label{thm:a_posteriori1}
    Let $u$ be the solution to the source 
    problem~\cref{eq:pb_rhs} and $u_N \in \Sobolev{2}(\mathbb{R}^d)$ be an approximate solution. There holds
   \begin{equation*}
       \|u-u_N\|_{\normA} \leq C^{1/2} \left[ \sum_{k=1}^{\K+1} 
       \|\sqrt{p_k} \Res(u_N)\|_{\dualAk}^2 \right]^{1/2}.
   \end{equation*}
 \end{theorem}
The error in energy norm is therefore bounded, up to a positive constant, to dual norms of localized residuals $\sqrt{p_k} \Res(u_N)$.
We will show in Section~\ref{subsec:practical_estimates} how to estimate the right-hand side with computable quantities. All the proofs are presented in Section~\ref{sec:proofs}.

\medskip

For simplicity of the presentation, we state the {\it a posteriori} error estimation for the eigenvalue problem in the case of simple eigenvalues only. Note however that clusters of eigenvalues could easily be incorporated to the analysis following~\cite{Cances2020-sw}. 
We start by stating an assumption on the gap between the computed eigenvalue and surrounding eigenvalues.
\begin{assumption}[Continuous-discrete gap conditions]
\label{as:gap_condition}
There exist $\underline{\lambda}_{i+1}\in\R$, and $\overline{\lambda}_{i-1}\in\R$ if $i>1$, which we take as $\overline{\lambda}_{i-1} =\lambda_{(i-1)N}$, such that there holds
\[
  \lambda_{i-1} \le \overline{\lambda}_{i-1} < \lambda_{iN} \;\; \text{ if $i>1$, } 
  \quad 
  \lambda_{iN} < \underline{\lambda}_{i+1} \le \lambda_{i+1}.
\]
\end{assumption}
\begin{remark}
    A possibility to avoid this assumption is to define $\lambda_i$ as the closest exact eigenvalue solution to \cref{eq:pb_eig} to the computed eigenvalue $\lambda_{iN}$ by
\begin{equation*}
     \lambda_i := \mathop{\rm argmin}_{\substack{\lambda \in \spectrum_{{d}}(A)}} |\lambda - \lambda_{iN}|,
\end{equation*}
and to define $\varphi_{i}$ as a corresponding eigenvector. However we prefer not to choose this definition because the variational principle ensuring that $\lambda_i \le \lambda_{iN}$ would be lost.
\end{remark}
We then define the gap constants
\begin{align}
     \label{eq:tildeCi}
     \tildeprefactor_i &:=
     {\rm min} \left( 
     \left(1 - \frac{\lambda_{iN}}{\overline{\lambda}_{i-1}}\right)^2,
     \left(1 - \frac{\lambda_{iN}}{\underline{\lambda}_{i+1}}\right)^2  \right),  
      \\
      \label{eq:tildeCihat}
     \hatprefactor_i &:=
     {\rm min} \left(
     \left(1 - 
     \frac{\lambda_{iN}}{\overline{\lambda}_{i-1}}\right)^2
     \overline{\lambda}_{i-1}, 
     \left(1 - 
     \frac{\lambda_{iN}}{\underline{\lambda}_{i+1}}\right)^2
     \underline{\lambda}_{i+1} 
     \right),
 \end{align}
 the left terms above being discarded for $i=1$.
With this definition, there holds
\begin{align*}
     \tildeprefactor_i & \le \min_{\substack{
     \lambda \in \spectrum(A) \backslash \{\lambda_i \}
     } 
     }
     \left(1 - \frac{\lambda_{iN}}{\lambda}\right)^2, \quad 
     \hatprefactor_i \le  \min_{\substack{
     \lambda \in \spectrum(A) \backslash \{\lambda_i \}
     } }\left(1 - 
     \frac{\lambda_{iN}}{\lambda}\right)^2\lambda.
 \end{align*}
\begin{remark}
\label{rem:gapi}
    In practical calculations, we use for simplicity the computed eigenvalue $\lambda_{(i+1)N}$ in place of $\underline{\lambda}_{i+1}$. Although not anymore fully guaranteed, the resulting error bounds barely change in practice.
    Rigorous bounds on the gap ensuring that \cref{as:gap_condition} could be however be obtained using (coarse) error bounds on the eigenvalues not necessitating gap information (e.g. Bauer--Fikke bound~\cite{Bauer1960-xe}).
\end{remark}

\begin{theorem}[$A$-error estimation for eigenvalue problem]
   \label{thm:a_posteriori2}
    For $1\leq i\leq N$, let 
    $(\lambda_{iN},\eigvec_{iN}) \in \mathbb{R} \times \Sobolev{2}(\mathbb{R}^d)$ be solution to~\cref{eq:pb_eig_Gal}, and 
    $(\lambda_i,\eigvec_i)$ be a solution to equation \cref{eq:pb_eig} 
    satisfying
    \cref{as:gap_condition} and
    $\langle \eigvec_i,\eigvec_{iN}\rangle_{\mathbb R^d}\geq 0$. 
    There holds
    \begin{equation}
        \|\eigvec_i-\eigvec_{iN}\|_{\normA} \leq 
        \left( C \tildeprefactor_i^{-1} r_i^2 +  \lambda_i C^2 \hatprefactor_i^{-2} r_i^4 \right)^{1/2},
        \label{eq:Aestimthm}
    \end{equation}
    where
    \[
        r_i^2 :=  \sum_{k=1}^{\K+1} \|\sqrt{p_k} \Res (\lambda_{iN},\eigvec_{iN})
            \|_{\dualAk}^2.
        \]
         Moreover, 
            \begin{equation}
        0 \leq \lambda_{iN} - \lambda_i \le C \tildeprefactor_i^{-1} r_i^2.
         \label{eq:eigenestimthm}
      \end{equation}
 \end{theorem}

The error on the eigenvector or the eigenvalue can again be bounded by terms involving localized residuals.

\subsection{Practical estimates}
\label{subsec:practical_estimates}

We now explain how the bounds in Theorem~\ref{thm:a_posteriori1} and Theorem~\ref{thm:a_posteriori2} can be estimated by practically computable quantities. 
The terms for $k=1,\ldots, K$ will be treated separately from the term involving $K+1$.

First, for $1 \leq k \leq \K$, let us focus on the evaluation of $\|\sqrt{p_k} v\|_{\dualAk}$ for $v \in \Ldeux(\mathbb{R}^d)$. For this, we use a spectral decomposition of $A_k$ which is easy to obtain as the spectrum of $A_k$ is discrete -- since $\Omega_k$ is compact -- and since $A_k$ is radially symmetric.
Let $(\varepsilon^{(k)}_j,\psi^{(k)}_j)\in\mathbb R\times \Sobolev{1} 
_0(\Omega_k)$, $j\in\mathbb N$, be the eigenpairs of $A_k$ such that $0<\varepsilon_1^{(k)} \leq \varepsilon_2^{(k)} \leq \dots$ counting multiplicities and 
\begin{equation*}
    A_k \psi^{(k)}_j = \varepsilon^{(k)}_j \psi^{(k)}_j, \quad \text{with} \quad 
    \langle \psi^{(k)}_j,\psi^{(k)}_{j'} \rangle_{\Omega_k} 
    = \delta_{jj'}, \quad j,j'\geq 1.
\end{equation*}
By spectral calculus, for any $v\in \Ldeux(\Omega_k)$ we have
$
   \displaystyle A_k^{-1}v = \sum_{j=1}^\infty\frac{1}{\varepsilon^{(k)}_j}\langle v, \psi^{(k)}_j\rangle_{\Omega_k} \psi^{(k)}_j,
$ thus 
\begin{equation*}
    \|v\|^2_{\dualAk} =
    \sum_{j=1}^\infty\frac{1}{\varepsilon^{(k)}_j}
    \left|\langle v, \psi^{(k)}_j\rangle_{\Omega_k}\right|^2.
\end{equation*}
As this sum is {\it a priori} not computable, we introduce the partial dual norm associated to the partial sum of finite degree $\J_k\in\mathbb N$ defined by
\[
    \mathcal I_k(v) := \mathcal I_k(v;\J_k) = \sum_{j=1}^{\J_k} \frac{1}{\varepsilon^{(k)}_j}
        \left|\langle v, \psi^{(k)}_j\rangle_{\Omega_k}\right|^2.
\]
We estimate the dual norm with respect to $A_k$ using this quantity as follows.

\begin{lemma}[Practical estimate]
    \label{lemma:uplow}
    For any $\J_k \in \mathbb{N}$ and $v\in \Ldeux(\Omega_k)$, there holds 
    \[
        \mathcal I_k(v) \leq \|v\|^2_{\dualAk}
        \leq
        \mathcal I_k(v)
        + \frac{1}{\varepsilon^{(k)}_{\J_k}}\left(\|v\|_{\Omega_k}^2 -
        \sum_{j=1}^{\J_k} 
        \left|\langle v,\psi^{(k)}_j\rangle_{\Omega_k}\right|^2\right).
    \]
\end{lemma}

\begin{remark}
    In all generality, 
    different parameters $J_k$ can be used for the different values of $k$ from 1 to $K$. In practice, we will use in the numerical simulations a single parameter for all $k=1,\ldots,K.$
\end{remark}

Thus for a given $v \in \Ldeux(\mathbb{R}^d)$, the estimation of $\|\sqrt{p_k} v\|_{\dualAk}$ can be done by first multiplying $\sqrt{p_k}$ and $v$, second by computing the lowest eigenpairs $(\varepsilon^{(k)}_j,\psi^{(k)}_j)_{j}$ of $A_k$ and finally computing $\mathcal I_k(\sqrt{p_k} v)$. 
Moreover Lemma~\ref{lemma:uplow} provides an error control on the approximation of $\|v\|_{\dualAk}$.

We now turn to the estimation of $\|\sqrt{p_{\K+1}} v\|_{\dualAMp}$
 for $v \in \Ldeux(\mathbb{R}^d)$. 
 This requires to compute the solution to a Poisson equation, which can be done
 using for example the Green’s function of the Laplace operator.

Our final estimates presented in Theorem~\ref{thm:practical1} and Theorem~\ref{thm:practical2} are fully guaranteed and practically computable.

\begin{theorem}[Guaranteed and practical estimation for source problem]
   \label{thm:practical1}
    Under the same assumptions as Theorem~\ref{thm:a_posteriori1}, for a fixed
    collection of natural numbers $\bm \J=(\J_k)_{1\leq k\leq \K}$, there holds
    \begin{equation*}
        \|u-u_N\|_{\normA} \leq C^{1/2} \widetilde r(\bm J),
    \end{equation*}
    where
    \begin{align*}
        \widetilde r(\bm J)^2 &:= \sum_{k=1}^\K
       \Bigg[ \mathcal I_k\left(\sqrt{p_k} \Res(u_N)\right) \\ &+
        \frac{1}{\varepsilon^{(k)}_{\J_k}}\Big(\|\sqrt{p_k}\Res(u_N)\|_{\Omega_k}^2 
        - \sum_{j=1}^{\J_k}
        \left|\langle \sqrt{p_k}\Res(u_N),\psi^{(k)}_j
        \rangle_{\Omega_k}\right|^2\Big)
        \Bigg]
        \\
        & + 
        \langle
        \sqrt{p_{\K+1}}\Res(u_N),(-\frac{1}{2} \Delta + \shift_{\K+1})^{-1} \sqrt{p_{\K+1}}\Res(u_N)\rangle_{\Omega_{K+1}}.
   \end{align*}
 \end{theorem}

\begin{theorem}[Guaranteed and practical estimation for the eigenvalue problem]
   \label{thm:practical2}
    Under the same assumptions as Theorem~\ref{thm:a_posteriori2}, for a fixed 
    collection of natural numbers $\bm \J=(\J_k)_{1\leq k\leq \K}$, there holds
    \begin{equation*}
        \|\eigvec_i-\eigvec_{iN}\|_{\normA} \leq 
        \left( C \tildeprefactor_i^{-1} \widetilde r_i(\bm J)^2 +  
        \lambda_{iN} C^2 \hatprefactor_i^{-2} {\widetilde r_i(\bm J)}^4 \right)^{1/2},
    \end{equation*}
    where 
    \begin{align*}
        \widetilde r_i(\bm J)^2 
        &:= \sum_{k=1}^\K 
        \Bigg[
        \mathcal I_k\left(\sqrt{p_k} \Res(\lambda_{iN},\eigvec_{iN})\right) \\ &+
        \frac{1}{\varepsilon^{(k)}_{\J_k}}\Big(\|\sqrt{p_k}\Res(\lambda_{iN},
        \eigvec_{iN})\|_{\Omega_k}^2 
        - \sum_{j=1}^{\J_k}
        \left|\langle \sqrt{p_k}\Res(\lambda_{iN},\eigvec_{iN}),\psi^{(k)}_j
        \rangle_{\Omega_k}\right|^2\Big) \Bigg]\\
        & + 
        \langle \sqrt{p_{\K+1}}\Res(\lambda_{iN},\eigvec_{iN}),
        (-\frac{1}{2} \Delta + \shift_{\K+1})^{-1} \sqrt{p_{\K+1}}\Res(\lambda_{iN},\eigvec_{iN})
        \rangle_{\Omega_{\K+1}},
   \end{align*}
   and
         \[
        0 \leq \lambda_{iN} - \lambda_i \le C \tildeprefactor_i^{-1} \widetilde r_i(\bm J)^2.
      \]
\end{theorem}

\subsection{Adaptive refinement strategy}

Since the discretization basis used is composed of atom-centered functions, for which there is in general no theoretical convergence rate (see \cite{GaussOpt} for a more detailed exposition), we aim at using our {\it a posteriori} error estimator composed of local atomic contributions to determine which atomic basis to enlarge. 
In that regard, for any atom indexed by $1\leq k\leq \K$, let us consider a spectral basis parameter $\J_k\in\mathbb N$ and define the local error indicator by
\begin{align}
    \label{eq:local_error_estimator}
    \eta_k^2 &:= 
    \mathcal I_{k}(\sqrt{p_k} \Res) +
    \frac{1}{\varepsilon^{(k)}_{\J_k+1}}\left(\|\sqrt{p_k}\Res\|_{\Omega_k}^2 
    - \sum_{j=1}^{\J_k}
    \left|\langle \sqrt{p_k}\Res,\psi^{(k)}_j
    \rangle_{\Omega_k}\right|^2\right),
\end{align}
where $\Res$ denotes either $\Res(u_N)$ or $\Res(\lambda_{iN},\eigvec_{iN})$ depending on the solved problem.

A natural refinement strategy consists in finding for which atom $k$ the local error indicator $\eta_k$ is the largest and increasing the number of basis functions centered on the $k^{\rm th}$ atom.
The corresponding algorithm is presented in \cref{algo:basis_adapt}.

\begin{remark}
    The adaptive refinement strategy does not require to compute the estimator related to the $(K+1)$-th component defined as
    \begin{equation*}
        \eta_{K+1}^2 := \langle \sqrt{p_{\K+1}}\Res(\lambda_{iN},\eigvec_{iN}),
        (-\frac{1}{2} \Delta + \shift_{\K+1})^{-1} \sqrt{p_{\K+1}}\Res(\lambda_{iN},\eigvec_{iN})
        \rangle_{\Omega_{\K+1}}.
    \end{equation*}
\end{remark}

\begin{algorithm}[t]
    \begin{algorithmic}[1]
        \REQUIRE{system of $\K$ atoms, atomic basis of size $n_k$ centered on atom $k$, $1\leq k\leq \K$.}
        \ENSURE{refined atomic basis of size $n'_k$, with
        $n'_k\geq n_k$ for all $1\leq k\leq \K$}

        \FOR{$k=1,\ldots,\K$}
            \STATE Evaluate error indicator $\eta_k^2$ defined in equation \cref{eq:local_error_estimator}.
        \ENDFOR
        \STATE Set $k_0=\arg\max_{1\leq k\leq \K}\, \eta_k^2$.
        \RETURN $n'_{k_0} = n_{k_0} + 1$ and $n'_k = n_k$ for all $1\leq
        k\leq \K, k\neq k_0$.
    \end{algorithmic}
    \caption{Adaptive refinement strategy for atom-centered basis sets.}
    \label{algo:basis_adapt}
\end{algorithm}

    \subsection{Computational cost of the error bounds}
    
    For each $k=1,\ldots,\K$ the computation of the corresponding term in the error bound of \cref{thm:practical2}  requires the spectral decomposition of the operator $A_k$, which is radially symmetric. For practical simulations, one can expect that large Gaussian-type orbital basis sets approximate correctly these radially symmetric operators and can be used as a reference basis. Once this is known, only scalar products have to be computed, which can be done very efficiently using numerical quadrature rules. In this context, the computation of the bounds is therefore very cheap. In the present section, we estimate the computational cost of these two steps in terms of floating-point operations.

    \paragraph{Atomic eigenproblems}

    The basis sets used to diagonalize each $A_k$ are expected to be more refined than the discretization basis centered on the same atom used to solve the molecular eigenproblem, in order to obtain a spectral decomposition of sufficient order. If $m_k$ denotes the size of the refined atomic basis per atom, diagonalizing the discretized operator $A_k$ costs $O(m_k^3)$ for each $k$, with upper bound $O(m^3\K)$ where $m=\max_k (m_k)$. 
    By translational invariance of the operators $(A_k)_{1 \leq k \leq K}$, the diagonalization has to be performed only once per atomic species, reducing the prefactor $\K$ to the number of different species in the molecule.  
    The remaining diagonalizations can also be done in parallel over different processors, thus improving the scaling.

    \paragraph{Numerical quadratures}

    All scalar products in the estimator can be evaluated through numerical quadratures, using a fixed number $n_q$ of quadrature points.
    For each $k=1,\ldots,\K$, the evaluation of the estimator $\eta_k$ is of the order $O(n_q m_kJ_kN)$ where $J_k$ is the number of atomic eigenfunctions and $N$ the number of atomic orbitals in the molecular basis. 
    Parallelization over the atomic positions or the quadrature points would improve this scaling.

\section{Numerical results}\label{sec:num_results_adapt}

In this section we present numerical simulations performed to illustrate the error bounds derived in the previous section, as well as the adaptive refinement strategy. 
The numerical results are organized into two parts. 
In a  first part we present experiments on a
1D toy model for which we can easily study the dependency of the error bounds with respect to different simulation parameters. In a second part we present simulations on a 3D model corresponding to standard electronic structure calculations with common basis sets used in the simulations of molecular systems.
All simulations are performed for diatomic systems, that is for $K=2$. For simplicity we use terms related to molecular simulations both in 1D and 3D, although the 1D toy model is not used in general for molecular simulations. In \cref{tab:params}, we summarize the problem parameters. 

In both cases we consider partition of unity functions that are radial and that are, in transition regions, (hashed on \cref{fig:omega})
based on the function defined on the 
interval $[a,b]$
\begin{equation}
\label{eq:partition_example}
    \forall a\leq x\leq b, 
    \quad p(x) = \frac{q(x-a)}{q(x-a) + q(b-x)}, 
    \quad q(x) = \exp(-1/x).
\end{equation}
This function increases on $[a,b]$, equals zero at $x=a$ and one at $x=b$. Moreover, $p$ is 
smooth on the closed interval $[a,b]$. It
can therefore be smoothly extended to a constant function.
A quick check ensures that the partition functions defined from this $p$ satisfy Assumption~\ref{assump:partition}.

\begin{figure}[t]
    \begin{center}
    \footnotesize{
        \begin{tabular}{ |c | l|}
            \hline
             $\shift$ & shift factor for Hamiltonian $A$ \\
             $\shift_1$ & shift factor for Hamiltonian $A_1$ \\
             $\shift_2$ & shift factor for Hamiltonian $A_2$ \\
             $\shift_3$ & shift factor for Hamiltonian $A_3$ \\
             $z_{-R}$    & atomic charge of atom 1 at position $-R$ \\
             $z_{R}$    & atomic charge of atom 2 at position $R$ \\
             $2R$     & interatomic distance \\
             $\ell$   & partition overlap width $0<\ell<2R$ \\
             $n_1$    & AO basis set size on atom 1 \\
             $n_2$    & AO basis set size on atom 2 \\
             $\bm J = (J_1,J_2)$ with $J_1=J_2$   & spectral basis size in \cref{lemma:uplow} \\ \hline
        \end{tabular}
        }
    \end{center}
    \captionof{table}{Summary of parameters for diatomic molecules with atoms at $-R$ and $R$.
    }
    \label{tab:params}
\end{figure}

\subsection{1D toy model}

We first present numerical results for a one-dimensional problem discretized with Hermite--Gaussian basis functions. 
The corresponding code is available in \cite{ape_hbs} for reproducing the simulation results.

\subsubsection{Numerical setting}

We denote by $z_{-R}>0$ and $z_{R}>0$ the nuclear charges of two 
atoms positioned at $-R$ and $R$ on the real line, respectively, for $R>0$.
We consider a soft Coulomb potential
with a softness parameter $\alpha>0$ defined by
\[
    V_{\alpha}(x) = \frac{1}{\sqrt{\alpha^2 + x^2}}.
\]
This potential is often used to mimic a Coulomb singularity for 1D simulations. The value of $\alpha$ is fixed at  $\alpha=0.5$ throughout the simulations. 

The considered Hamiltonian operator on $\R$ is then 
\begin{equation*}
    A = -\frac{1}{2} \Delta - z_{-R} V_\alpha(x+R) - z_{R}V_\alpha(x-R) + \shift, 
\end{equation*}
for some $\shift\in\R$.
We will consider two problems: 

\ \noindent 1. the source problem for
    \begin{align*}
        f = & \Big(-\frac{1}{2} \Delta - z_{-R} V_\alpha(\cdot+R) + \shift_1 \Big)h_1(\cdot+R)  + \Big(-\frac{1}{2} \Delta - z_{R} V_\alpha(\cdot-R) + \shift_2 \Big)h_1(\cdot-R),
    \end{align*}
    where $h_1$ is defined below in~\cref{eq:hermite}.
    
\ \noindent 2.     the eigenvalue problem $A \eigvec_1 = \lambda_1 \eigvec_1$ for the lowest eigenvalue.
\\

As there is no analytical expression for the exact solution to both problems, a reference solution is obtained with a
 finite difference scheme with a uniform grid of $N_g=2001$ points over the bounded 
computational domain $\Omega = [-5R, 5R]$.

The partition of unity is defined as follows.
Given $0<a_{\mathrm{min}}< a_{\mathrm{max}} \leq 2R$, the computational domain $\mathbb R$ is decomposed into $\Omega_1,\Omega_2,\Omega_3$, defined symmetrically for both atoms as
\begin{align*}
    \Omega_1 &= [-R-a_{\mathrm{max}},-R+a_{\mathrm{max}}], \quad \Omega_2 = [R-a_{\mathrm{max}},R+a_{\mathrm{max}}]  \\
    \Omega_3 &= (-\infty, -R-a_\mathrm{min}]\cup[ R+a_\mathrm{min}, \infty),
\end{align*}
with $-R+a_{\rm min} = R- a_{\rm max}$, ensuring that 
$p_1(x)+p_2(x)=1$ for $x \in [-R+a_{\mathrm{min}},R-a_{\mathrm{min}}]$, thus the overlap width has length $\ell=2(R-a_{\mathrm{min}})$.
Such partition of unity based on $p$ in \cref{eq:partition_example} is illustrated in \cref{fig:partition}.

\begin{figure}[t]
    \begin{center}
        \includegraphics[width=\linewidth]{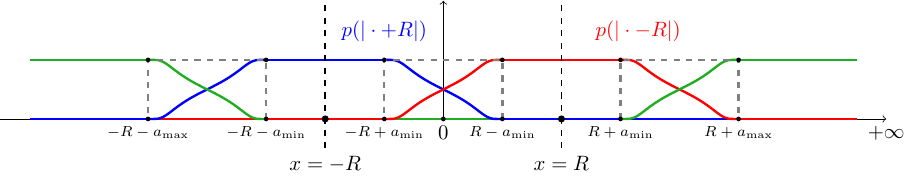}
        \caption{Partition of unity for two atoms placed at $x=\pm R$ in one
        dimension. Partition functions associated to atom
        $(z_{-R},-R)$ (in blue), to atom
        $(z_{R},R)$ (in red) and to the complementary domain $\Omega_3$ (in green). 
        }
        \label{fig:partition}
    \end{center}
\end{figure}

The considered AO basis set is the Hermite--Gaussian basis set \cite{GaussOpt} of functions defined for any $n\in\mathbb N$ as
\begin{equation}
\label{eq:hermite}
        \forall x\in\mathbb R, \quad
    h_n(x) = c_n q_n(x) \exp(-x^2/2),
\end{equation}
where $q_n$ is the Hermite polynomial of degree $n$ and $c_n$ a normalization
constant such that $\int_{\mathbb R}h_n^2 = 1$. 
In practice, we consider $n_1$ (resp. $n_2$) as the basis set size on atom 1 (resp. 2), so that the corresponding basis is $\{h_n(\cdot + R)\}_{0\le n \le n_1-1}$ (resp. $\{h_n(\cdot - R)\}_{0\le n \le n_2-1}$).
Note that $\{h_n\}_{n\in\mathbb
N}$ is an orthonormal basis of $\Ldeux(\mathbb R)$ so that the atomic basis sets are systematically improved when $n_1$ and $n_2$ are increased. However, linear dependencies appear between overlapping basis sets centered on different nuclei, a phenomenon commonly refered to as overcompleteness.
 
\subsubsection{Practical computation of the error estimator} \label{sec:practical_compute_1D}

The error estimators of \cref{thm:practical1} and \cref{thm:practical2} require the
evaluation of several involved quantities, that we obtain in this setting as follows.

\paragraph{Approximate solutions}
   They are obtained by solving either a linear system or a generalized eigenvalue problem on a Hermite--Gaussian basis. We eliminate basis set linear dependencies using pivoted Cholesky decomposition in order to reduce the condition number of the overlap matrix.

\paragraph{Constant $C$} 
    The derivatives of partition of unity functions can be evaluated analytically using forward differentiation (e.g. ForwardDiff \cite{RevelsLubinPapamarkou2016} in Julia). The maximization is then performed on the overlap region of supports only, using a golden-section method.

\paragraph{Constant $c_A$} 
    The optimal constant that we use in the sensitivity analysis below is $c_A = 1/\sqrt{\lambda_1}$ where $\lambda_1$ is the lowest eigenvalue of the reference Hamiltonian. For the error bounds, since $\lambda_1$ is not known, we use instead that the potential
    is lower bounded by a constant \underline{$V$} so that $c_A$ can be taken as $1/\sqrt{\underline{V}+\varsigma}$ provided that $\underline{V}+\varsigma>0$.

\paragraph{Gap constants $\hatprefactor_1$ and $\tildeprefactor_1$}
    As mentioned in \cref{rem:gapi} we take the approximate eigenvalue $\lambda_{2N}$ as $\underline{\lambda}_{2}$. We also compute guaranteed gap constants using the lower bound obtained by Weyl inequality \cite{Weyl1912DasAV} of the form $\lambda_1(-\frac 14 \Delta +V_1 +\frac{\shift}{2}) + \lambda_2(-\frac 14\Delta +V_2 + \frac{\shift}{2})\leq \lambda_2(-\frac 12\Delta +V_1+V_2+\shift)$. Taking this lower bound as $\underline{\lambda}_{2}$ numerically verifies \cref{as:gap_condition} in our setting. Moreover, the bound is computable by diagonalizing the radially symmetric operators on the reference finite difference grid. The lower bound we use in the sensitivity analysis below is $\underline{\lambda}_{2}=\lambda_2$ where $\lambda_2$ is the second lowest eigenvalue of the reference Hamiltonian.
    Of course since the second eigenvalue $\lambda_{2N}$ used in the nonguaranteed error bound is larger than the one used for guaranteed error estimates,
    and 
    since the 
    gap constants $\hatprefactor_1$ and $\tildeprefactor_1$ 
    are increasing with respect to
    $\underline{\lambda}_2$ when $\lambda_{1N}<\underline{\lambda}_2$, we expect the inverse of gap constants to be larger in the guaranteed case than in the nonguaranteed case. This will be illustrated in Figure~\ref{fig:vary_overlap}.
 
\paragraph{Spectral basis} 
    We solve atomic problems to find $\psi^{(k)}_i$ for $k=1,2$. Atomic operators are discretized on the same finite difference grid as the reference Hamiltonian. This resolution can be precomputed once for all internuclear configurations.

\paragraph{Dual atomic operator norms}
    They are approximated following 
     \cref{lemma:uplow} with some parameter $\bm J$.

\paragraph{Dual Laplacian norm}
     A shifted Poisson problem of the form $(-\frac 12\Delta + \shift_3)w=\sqrt{p_3}\Res$ is solved with finite differences.

\subsubsection{Quality of the error estimator}

We compute the practical error estimators given by Theorems
\ref{thm:practical1} and \ref{thm:practical2} respectively for the source and the eigenvalue problems and compare it to the exact error in $A$-norm on \cref{fig:vary_overlap_src} and \cref{fig:vary_overlap}, taking in both cases $R=1$, $z_{-R}=z_{R}=1$, $\shift=4$, $\shift_1=\shift_2=3$, $\shift_3=1$, $\bm J=(17,17)$. First we observe in both cases that the {\it a posteriori} bound is an upper bound of the error. Second it closely captures the behavior of the error for both problems, regardless of the choice of the partition overlap parameter $\ell$,
with tight estimates for larger overlap sizes, as explained in the sensitivity analysis below. Third the guaranteed bound is less tight and higher than the non-guaranteed bound by approximately one order of magnitude.

\begin{figure}[t]
    \centering
    \includegraphics[width=0.48\textwidth]{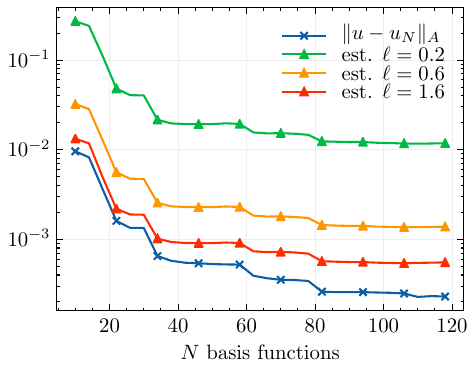}
    \caption{Practical estimates of the approximation error $\|u-u_N\|_A$ for $u$ solution to~\cref{eq:pb_rhs}, using \cref{thm:practical1}, for various overlap width parameter $\ell$ values. Fixed particles in 1D with $R=1$, $z_{-R}=z_{R}=1$, $\shift=4$, $\shift_1=\shift_2=3$, $\shift_3=1$, $\bm
J=(17,17)$. On the $x$-axis is the number
    of AO basis functions equal to $N=n_1+n_2$ with $n_1=n_2=N/2$ per atom.} 
    \label{fig:vary_overlap_src}
\end{figure}

\begin{figure}[t]
    \centering
    \begin{subfigure}[T]{0.48\textwidth}
        \centering
        \includegraphics[width=\textwidth]{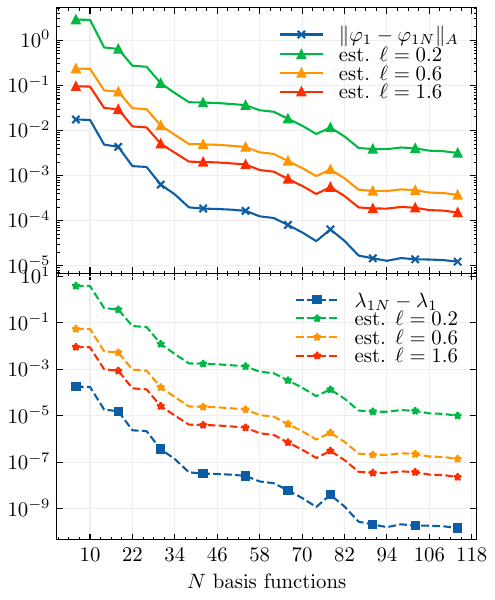}
        \caption{Using non-guaranteed gap constants.}
        \label{fig:non-guaranteed}
    \end{subfigure}\hfill
    \begin{subfigure}[T]{0.48\textwidth}
        \centering
        \includegraphics[width=\textwidth]{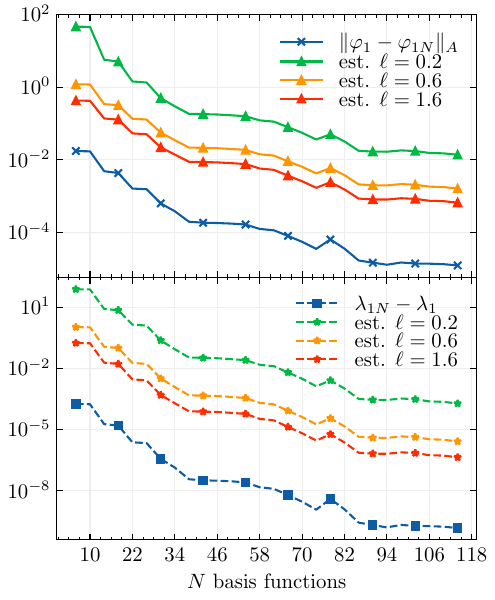}
        \caption{Using guaranteed gap constants.}
        \label{fig:guaranteed}
    \end{subfigure}

    \caption{Practical non-guaranteed (left) and guaranteed (right) estimates of the approximation error $\|\eigvec_1 - \eigvec_{1N}\|_A$ (solid lines) and $\lambda_{1N} - \lambda_1$ (dashed line) for 
        first eigenpair $(\lambda_1,\eigvec_1)$ of~\cref{eq:pb_eig}, using \cref{thm:practical2}, for various overlap width parameter $\ell$ values. Fixed particles in 1D with $R=1$, $z_{-R}=z_{R}=1$, $\shift=4$,
    $\shift_1=\shift_2=3$, $\shift_3=1$, $\bm
J=(17,17)$. On the $x$-axis is the number
    of AO basis functions equal to $N=n_1+n_2$ with $n_1=n_2=N/2$ per atom.} 
    \label{fig:vary_overlap}
\end{figure}

\subsubsection{Numerical parameter sensitivity}\label{sec:sens_1D}

We now study the influence of several numerical parameters on the error bounds for the eigenvalue problem. The results are similar for the source problem therefore we do not display them here. 

\paragraph{Spectral basis size}

We first illustrate in
\cref{fig:spectral_K} the 
influence of the spectral basis size $\bm J$ from \cref{lemma:uplow} on
the dual norm estimation. Results show that the upper and lower bounds
become tight when increasing truncation degree, which is expected as the
spectral approximation systematically improves.

\begin{figure}[t]
    \centering
    \includegraphics[width=0.7\textwidth]{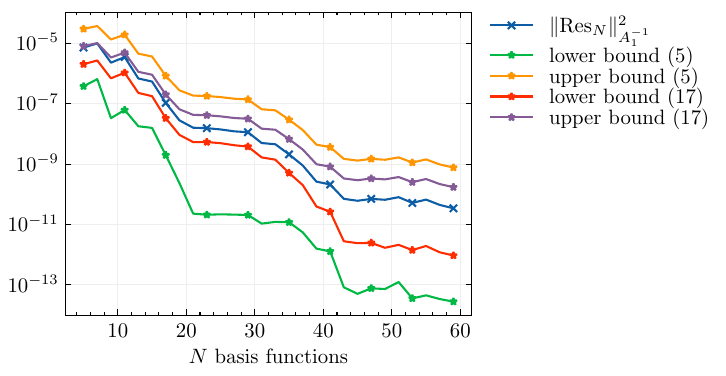}
    \caption{Upper and lower bounds in \cref{lemma:uplow}
    for estimating $\|\Res(\lambda_{1N},\varphi_{1N})\|^2_{A_1^{-1}}$ (in blue), using $\bm J = (5,5)$ or $\bm J = (17,17)$. Fixed $R = 1$, $z_{-R}=z_{R}=1$, $\shift=4$, $\shift_1=\shift_2=3$, $\shift_3=1$, $\ell=0.8$. 
    }
    \label{fig:spectral_K}
\end{figure}

\paragraph{Overlap size}

The influence of the size of the overlap width in the partition of unity given by $\ell$ is shown in
\cref{fig:vary_overlap_src} and \cref{fig:vary_overlap}. As expected the estimator is tighter for
larger overlaps. Indeed the constant $C$ defined in equation \cref{eq:C} for the
practical estimator depends on $\ell$ and on the shift factors. For fixed
shift factors, the constant $C$ increases as the overlap size decreases, 
as shown in \cref{tab:params_sens_C}.

\paragraph{Shift values}

In \cref{tab:params_sens_C}
we also observe that the constant $C$ decreases when the shift values decrease.
However the shift values cannot be decreased too much
without breaking \cref{as:posdef} on positive-definiteness of atomic
operators.

\paragraph{Gap constants}

Finally we study the gap constants $\gamma_1$ and $\hatprefactor_1$ defined in \cref{eq:tildeCi} and \cref{eq:tildeCihat}.
In practice, the magnitude of the final estimator in \cref{thm:a_posteriori2} increases as the values of gap constants decrease. It is thus desirable to choose a shift value $\shift$ that minimises $1/\gamma_1$. Optimal $1/\gamma_1$'s closer to zero can
be obtained by selecting smaller values of the shift factor, as seen in
\cref{tab:params_sens_gap}. At the same time, however, smaller values of
$\shift$ lead to sub-optimal larger values of the constant $\cA$ in
\cref{eq:cA}, thus in larger values of the constant $C$ \cref{eq:C}, the common prefactor of the {\it a posteriori} estimates.
As a conclusion, there exist no optimal value of $\shift$ that minimises both $1/\gamma_1$ and $c_A$ at the same time. 

\begin{table}[t]
    \centering
    \begin{subtable}[t]{0.54\textwidth}
        \begin{center}
        {\footnotesize
        \begin{tabular}{ |l | r | r | r | r |}
            \hline
             \diagbox[width=\dimexpr \textwidth/8+2\tabcolsep\relax,
             height=0.8cm]{$\ell$}{$\shift_a$}  
                      & 5.0      &  3.0         & 2.0         & 1.0 
             \\ \hline
             $0.1$    & 514.14 & 513.32 & 512.91 &  512.50 \\
             $0.3$    &   8.07 &   7.25 &   6.84 &    6.43 \\
             $0.5$    &   2.71 &   1.89 &   1.48 &    1.16 \\
             $0.8$    &   2.16 &   1.34 &   1.00 &    1.00 \\
             $0.9$    &   2.09 &   1.27 &   1.00 &    1.00 
             \\ \hline
        \end{tabular}
        }
        \end{center}
        \caption{Values of the constant $C$ defined in equation \cref{eq:C} for varying overlap 
        width $2\ell$ and atomic shifts $\shift=4$, $\shift_1=\shift_2=\shift_a$.
        }
    \label{tab:params_sens_C}
    \end{subtable}\hfill
    \begin{subtable}[t]{0.38\textwidth}
        \begin{center}    
        \footnotesize{
        \begin{tabular}{ |l | r | r |}
            \hline
             $\shift$  & $1/\gamma_1$ & $c_A$
             \\ \hline
             $3.0$    & 15.59  & 0.83  \\
             $4.0$    & 35.97  & 0.64  \\
             $5.0$    & 64.73  & 0.53  \\
             $6.0$    & 101.90  & 0.47  \\
             $7.0$    & 147.45  & 0.42  \\
             $8.0$    & 201.40  & 0.39
             \\ \hline
        \end{tabular}
        }
    \end{center}
    \caption{Values of  $1/\gamma_1$, with $\gamma_1 = \min_{\substack{
     \lambda \in \spectrum(A) \backslash \{\lambda_1 \}}} \Big(1-\frac{\lambda_{1N}}{\lambda}\Big)^2 $  and of $c_A$ in \cref{eq:cA}, for varying shift factor
    $\shift$. Fixed $\shift_1=\shift_2 = 1$.
    }
    \label{tab:params_sens_gap}
    \end{subtable}
    \caption{Computed parameter values (up to first decimal digits). Fixed common parameters: $R=1$, $z_{-R}=z_{R}=1$, $\shift_3 = 1$.}
\end{table}

\subsubsection{Adaptive basis sets}

In \cref{fig:adapt} we display the performance of the adaptive basis
generation of \cref{algo:basis_adapt} for two sets of parameters, one with $z_{-R}=3, z_R = 1$ and the other with $z_{-R}= z_R = 1$. 

In the first case, in the non-adaptive strategy, $n_1$ and $n_2$ are incremented such that $n_1 = n_2 = \frac{N}{2}$. 
In the adaptive refinement strategy, we observe that the algorithm varies $n_1$ from 5 to 53 and fixes $n_2=5$.
For small basis sizes below $N=20$, the adaptive basis strategy has little effect. However, for a larger basis size such as $N=60$, we gain about one order of magnitude by using our 
refinement strategy, with the same number of basis functions. 
Finally, when $z_{-R}=z_R$, in \cref{fig:adapt_identical} we see the performance of the
refinement strategy. The refinement strategy globally reproduces the behavior of the
optimal symmetric case $n_1=n_2$ for identical charges, which is expected. We have observed that changing the simulation parameters $\ell$ and $J$ barely change the numerical results. This is in accordance with the plots in~\cref{fig:vary_overlap} and~\cref{fig:spectral_K}, where these parameters only shift the convergence curves but do not alterate the behavior.

\begin{figure}[t]
    \centering
    \begin{subfigure}[T]{0.49\textwidth}
        \centering
        \includegraphics[width=.9\textwidth]{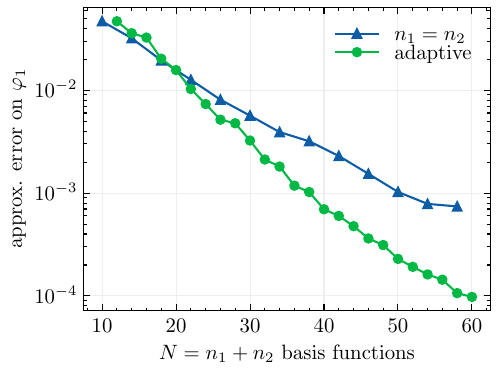}
        \caption{Distinct nuclear charges $z_{-R}=1$, $z_{R}=3$.}
        \label{fig:adapt}
    \end{subfigure}\hfill
    \begin{subfigure}[T]{0.49\textwidth}
        \centering
        \includegraphics[width=.9\textwidth]{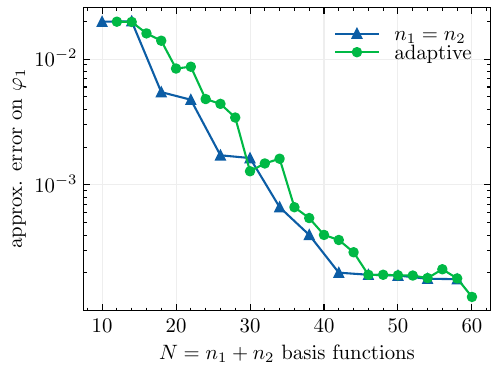}
        \caption{Identical nuclear charges $z_{-R}=z_R=1$.}
        \label{fig:adapt_identical}
    \end{subfigure}

    \caption{Adaptive basis sets on molecules in 1D. Exact operator norm error $\|\eigvec_1 - \eigvec_{1N}\|_A$ for adaptive basis set (in green) versus uniformly refined one with $n_1=n_2$ (in blue). 
    Common fixed parameters: $R= 1$, $\shift=7$, $\shift_1=\shift_2=4$, $\shift_3 = 1$, $\bm J=(17,17)$, $\ell=0.3$.}
\end{figure}

\subsection{3D diatomic molecules}

We now consider a more realistic setting in chemistry, 
addressing Gaussian-type orbital (GTO) basis sets.
The corresponding code, available in \cite{ape_gto}, reproduces all 
the plots. All numbers are given in atomic units.

\subsubsection{Numerical setting}

We still denote by $z_{-R}>0$ and $z_{R}>0$ the nuclear charges of two 
atoms positioned at $-R$ and $R$ on the $x$-axis, respectively, for $R\in \mathbb{R}^3$.
We consider the Coulomb potential, $V_C(x) = |x|^{-1}$ for any $x\in\mathbb R^3$. 
The Hamiltonian operator is then
\begin{equation*}
    A = - \frac 12\Delta - z_{-R} V_C(x + R) - z_R V_C(x- R) + \shift,
\end{equation*}
for a shift $\shift\in\R$.
For this Hamiltonian, we consider the eigenvalue problem \cref{eq:pb_eig} for the lowest eigenvalue $\lambda_1$, discretized using a basis of GTOs centered at nuclear positions. 
We will consider two molecular systems, namely \ce{H_2^+} ($z_{-R} = z_R = 1$) and \ce{LiH^{3+}} ($z_{-R} = 3, z_R = 1$).

In order to assess the discretization error, we use as a reference a finite element resolution in a large enough computational box, obtained by the HelFEM program \cite{Lehtola_2019_review, Lehtola_code}. The spectral bases of the atomic Hamiltonians are also obtained by solving atomic problems using the same code.
The main motivation for using HelFEM is that, unlike other programs, it provides a quadrature rule for accurately computing three-dimensional integrals involving a Coulomb potential singularity. 
HelFEM achieves this by transforming Cartesian coordinates to prolate spheroidal coordinates, enabling to transform the singular Coulomb integral into a smooth one~\cite{Lehtola_2019}.

\subsubsection{Practical computation of the error bounds}

We describe below the computation of the quantities involved in the error estimates.

\paragraph{Eigenvalue problem} 
    The approximate solution is obtained using the PySCF Python module \cite{pyscf}, by solving an unrestricted Hartree--Fock problem on a GTO basis set. 
    The reference solution is obtained using finite elements with HelFEM~\cite{Lehtola_2019} in a radial box of size $[-40,40]$. 

\paragraph{Constant $C$} 
    The partition of unity being defined from radial functions on the atom-centered balls, 
    we evaluate its gradient and Laplacian in spherical coordinates. 
    Since $K=2,$ the constant $C$ can be computed by minimizing a one-dimensional function. 
    In practice, this function is evaluated 
    on a set of 500  uniformly distributed points on the interval $[-2R,2R]$.

\paragraph{Constant $c_A$} As in the 1D case, the constant $c_A$ can be taken as an upper bound of $1/\sqrt{\lambda_1}$, that is we need a lower bound $\underline{\lambda_1}$ for $\lambda_1$, the lowest eigenvalue of the exact Hamiltonian. For this, we use the decomposition~\cref{eq:atomic_decomp} and bound the eigenvalue by the sum of atomic eigenvalues which are explicitely known~\cite[Theorem 10.10]{Teschl2014-pl}. 
For $H_2^+,$ we take $\underline{\lambda_1} = -2+\varsigma$. For \ce{LiH^{3+}} we take $\underline{\lambda_1} = -8+\varsigma$.

\paragraph{Gap constants $\hatprefactor_1$ and $\tildeprefactor_1$}
    As mentioned in \cref{rem:gapi} we take the approximate eigenvalue $\lambda_{2N}$ as $\underline{\lambda}_{2}$.

\paragraph{Spectral basis} 
    We solve the atomic eigenvalue problem on $A_k$ to find $\psi^{(k)}_i$ for $k=1,2$. Atomic operators are discretized on a finite element basis over a bounded computational box defined from the atom using HelFEM. The resolution can be entirely precalculated once for all internuclear configurations. 
    
\paragraph{Dual atomic operator norm}
    For $k=1,2$, we use \cref{lemma:uplow} on a spectral basis of size $\bm J=(J_1,J_2)$, $J_1=J_2$ controlled by maximal angular momentum equal to 6 for \ce{H2+} (that is $J_1 = 36$) and equal to 10 for \ce{LiH^{3+}} (that is $J_1 = 100$). Note that we did not select larger $J_k$ because the size of the spectral basis is limited by the Lebedev quadrature errors in PySCF, that tend to increase with angular momenta. Typically, these errors are of order \num{e-4} for spherical harmonics of angular momenta greater than 7 with a Lebedev quadrature of order 13.

\paragraph{Dual Laplacian norm}

A term that requires a careful treatment is the dual operator norm of the shifted Laplacian operator, used in \cref{thm:practical2}. We compute this by directly evaluating the Green's function that yields a 6-dimensional integral. 
In its most general form, the quantity to evaluate boils down to the following integral, for $a=\sqrt{\shift_{M+1}}$, then transformed by a change of variables $z=x-y$ to
\[
    \iint_{\mathbb R^3\times\mathbb R^3}\frac{f(x)f(y)}{|x-y|}e^{-a|x-y|}\,\dd{x}\dd{y} 
    = \int_{\mathbb R^3} f(y) \left( \int_{\mathbb R^3} \frac{f(z+y)}{|z|}e^{-a|z|}\,\dd{z} \right)\,\dd{y}.
\]
To accurately evaluate the integral, we use the quadrature rule provided by HelFEM to alleviate the Coulomb singularity.

\paragraph{Numerical integration} 

The quadrature rule provided by HelFEM allows to integrate any function over the computational domain $\Omega$ by projecting it on the finite element basis. This allows to directly compute the operator norm $\|\cdot\|_{A}$.
For evaluating $\|\cdot\|_{A_k}$ norm for $k=1,2$ we take advantage of the spherical symmetry, evaluate the radial integral using the grid provided by HelFEM and the integral on the sphere using a Lebedev quadrature rule of order 13.

\subsubsection{Quality of the error estimates}

On \cref{fig:estimator_3d} we plot the the practical error estimator together with the exact error in the case of a molecule \ce{H2+} for many GTO basis sets from the literature. 
We observe that the error estimation closely follows the behavior of the exact error. 
For large basis sets we observe some plateau of the error estimate while the error on the eigenfunctions continues to decrease. There are several possible explanations, in particular the chosen rank $\bm J$ of the correction in \cref{lemma:uplow} which may become too low, or numerical integration parameters.
Further note that the term corresponding to the Laplace inversion (that is in red on \cref{fig:estimator_3d}) decreases sharply for large basis sets, which also corresponds to basis sets with diffuse functions, which is indeed consistent.

\begin{figure}[t]
    \centering 
    \begin{subfigure}[T]{0.5\textwidth}
        \centering
        \includegraphics[width=\textwidth]{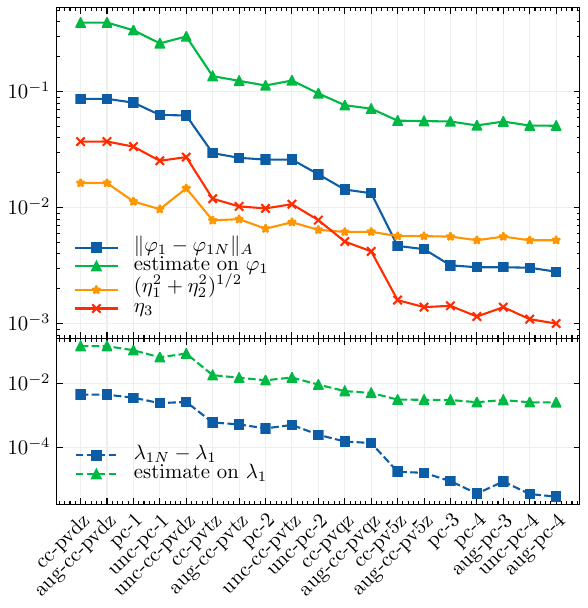}
        \caption{
        Estimate of \cref{thm:practical2} and local contributions, for \ce{H2+} molecule. Fixed $R=0.7$, $z_{-R}=z_R=1$, $\shift=4$. On the $x$-axis is the AO basis.%
        }%
        \label{fig:estimator_3d}
    \end{subfigure}\hfill
    \begin{subfigure}[T]{0.42\textwidth}%
        \centering%
        \includegraphics[width=\textwidth]{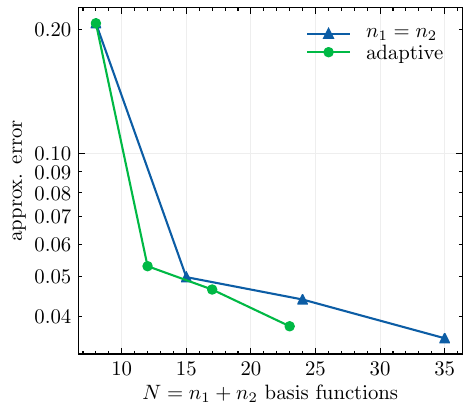}
        \caption{
            Comparison of adaptive basis set 
            versus uniformly refined one with $n_1=n_2$, for \ce{LiH^{3+}} molecule. 
            Fixed $R= 2$, $z_{-R}=3$, $z_R=1$, $\shift=9$.
            }
        \label{fig:adapt_3d}
    \end{subfigure}
    \caption{Approximation errors $\|\eigvec_1-\eigvec_{1N}\|_A$ (and $\lambda_{1N} - \lambda_1$) for diatomic molecules in 3D. Fixed common parameters: $\shift_1=\shift_2 = 
        \shift_3 = 3$, $a_{\min} = 0.1, a_{\max}=0.8$. }
\end{figure}

\subsubsection{Adaptive AO basis}

Finally we test the adaptive basis refinement of \cref{algo:basis_adapt}, using as a pool the cc-pVXZ basis sets, for X$\in\{$D,T,Q,5$\}$, for which increasing the basis size decreases the discretization error. 
Numerical results on \cref{fig:adapt_3d} show that our error indicator detects the atom whose basis is more preferable to refine. 
Indeed, with the same number of basis functions we arrive at better approximations of ground-state energy than if both atoms were refined. 
In particular, using the adaptive algorithm, only the basis of the Lithium atom is refined, while the Hydrogen basis is kept fixed and equal to cc-pVDZ.
Let us note that, precomputing the atomic spectral bases by diagonalization of atomic operators using HelFEM takes about 1 second, while the final computation of atomic estimators $\eta_1^2$ and $\eta_2^2$ only takes approximately 2 seconds for both Lithium and Hydrogen atoms on a desktop computer; 
the estimate $\eta_3^2$ which appears in the global error estimate does not need to be computed for adaptive refinement.
This has to be compared to the computation time necessary to obtain the approximate solution for \ce{LiH^{3+}} using PySCF, which is around 1 second, hence a similar time. Note however that the precomputation with HelFEM has to be done only once per atom type, and does not depend on the nuclei configuration. Moreover, the estimators $\eta_k$ can be easily computed in parallel, hence we expect the overhead computation time to be way smaller for large systems in comparison to the computation of the solution to the eigenvalue problem.
 
\section{Proofs}
\label{sec:proofs}

In this section, we present the proofs of the results presented in~\cref{sec:results}.

\subsection{Preliminaries}

We start by presenting two preliminary lemmas.

\begin{lemma} There holds
   \label{lemma:lem1}
    \begin{equation}
    \label{eq:5.1}
        \forall v\in \Sobolev{1}(\mathbb R^d), 
        \quad \left( \sum_{k=1}^{\K+1} \| \sqrt{p_k} v\|_{\normAk}^2 
        \right)^{1/2} \leq C^{1/2} \|v\|_{\normA},
    \end{equation}
    with $C$ being the positive constant defined in~\cref{eq:C}.
\end{lemma}

\begin{proof} 
For $v \in \Sobolev{1}(\mathbb R^d)$, by Assumption~\ref{assump:partition} on $p_k$, $\sqrt{p_k}v \in \Sobolev{1}_0(\Omega_k)$ for $k=1,\dots,\K$. By definition of the energy norm restricted to the subdomain $\Omega_k$, 
    for $k=1,\ldots,\K$, one has the equality
    \[
       \| \sqrt{p_k} v\|_{\normAk}^2  
        = \frac{1}{2}\langle \nabla(\sqrt{p_k} v) , \nabla(\sqrt{p_k} v)\rangle_{\Omega_k} + 
        \langle \sqrt{p_k} v, (V_k + \shift_k) (\sqrt{p_k} v) 
       \rangle_{\Omega_k}.
    \]
    Moreover, note that 
    \begin{align*}
        \|\sqrt{p_{\K+1}}v\|_{\normA_{\K+1}}^2 &= \frac{1}{2} \langle \nabla(\sqrt{p_{\K+1}}v),\nabla(\sqrt{p_{\K+1}}v) \rangle_{\mathbb R^d} + \langle \sqrt{p_{\K+1}}v, \shift_{\K+1}\sqrt{p_{\K+1}}v \rangle_{\mathbb R^d}.
    \end{align*}
    Expanding the differential operator $\nabla$ gives for $k=1,\dots,\K$
    \begin{align*} 
        \langle \nabla(\sqrt{p_k} v) , \nabla(\sqrt{p_k} v)\rangle_{\Omega_k} &= 
        \int_{\Omega_k} \left|\sqrt{p_k} \nabla v +
        \frac{v}{2\sqrt{p_k}}\nabla p_k\right|^2 \\
        &= 
        \int_{\Omega_k} p_k |\nabla v|^2 
        +  \int_{\Omega_k}  (\nabla p_k) \cdot (v \nabla v)
        + \int_{\Omega_k} \frac{|\nabla p_k|^2}{4 p_k} v^2 \\
        &=  \int_{\mathbb R^d} p_k |\nabla v|^2 
        -\frac12 \int_{\mathbb R^d} (\Delta p_k) v^2
        + \int_{\mathbb R^d} \frac{|\nabla p_k|^2}{4 p_k} v^2,
    \end{align*}
    where in the last line we integrated by parts, used that $p_k$ is supported in $\Omega_k$ and set $\frac{|\nabla p_k|^2}{4 p_k}=0$ outside of $\Omega_k$.
    Similarly, we have that 
    \[
    \langle \nabla(\sqrt{p_{\K+1}}v),\nabla(\sqrt{p_{\K+1}}v) \rangle_{\mathbb R^d}=  \int_{\mathbb R^d} p_{\K+1} |\nabla v|^2 
        -\frac12 \int_{\mathbb R^d} (\Delta p_{\K+1}) v^2
        + \int_{\mathbb R^d} \frac{|\nabla p_{\K+1}|^2}{4 p_{\K+1}} v^2.
    \]
    Thus evoking the property of equation \cref{eq:poiprop} of the partition 
    of unity, we obtain
    \begin{equation} \label{eq:3_lem1}
       \sum_{k=1}^{\K+1} \| \sqrt{p_k} v\|_{\normAk}^2 
       = 
       \frac{1}{2} \|\nabla v\|^2_{\mathbb R^d}
        + \sum_{k=1}^{\K+1}
        \int_{\mathbb R^d} \left(-\frac 14 \Delta p_k 
        + \frac{|\nabla p_k|^2}{8 p_k}
        + (V_k + \shift_k) p_k\right) v^2,
    \end{equation}
    taking $V_{K+1} = 0$.
    The energy norm of $v$ squared admits the expression
    \[
        \|v\|^2_{\normA} = \frac{1}{2}\|\nabla v\|^2_{\mathbb R^d} +
        \sum_{k=1}^{\K+1} \langle v,V_kv\rangle_{\mathbb R^d} + 
        \shift \|v\|^2_{\mathbb R^d}.
    \]
    Adding and subtracting this quantity to equation \cref{eq:3_lem1}, then majorizing 
    the integral over the whole space, one obtains the estimate
    \begin{align*}
       \sum_{k=1}^{\K+1} \| \sqrt{p_k} v\|_{\normAk}^2 
       \!\!\!
        &= \|v\|_{\normA}^2  
            + \sum_{k=1}^{\K+1}
            \int_{\mathbb R^d}\left(-\frac 14\Delta p_k 
            + \frac{|\nabla p_k|^2}{8 p_k}
            + V_k(p_k - 1)  
            + (\shift_k - \shift) p_k \right) v^2 \\
        &\le
        \|v\|_{\normA}^2 + \|v\|_{\mathbb R^d}^2
            \sup_{\mathbb R^d} 
            \left( \sum_{k=1}^{\K+1} - \frac14 \Delta p_k + 
            \frac{|\nabla p_k|^2}{8 p_k} + V_k (p_k-1)  
           \! + \! (\shift_k - \shift) p_k \!\! \right)^{\!\!\!\! +} \!\!\!,
   \end{align*}
   where 
   for $a \in \R$, $a^+:=\max
   \{a,0\}$
   .
   Lastly
   evoking 
   \cref{eq:cA}, we prove the desired result.
\end{proof}

\begin{lemma}\label{lemma:lem2} 
There holds
   \[
       \forall v\in \Ldeux(\mathbb R^d),\quad 
       \|v\|_{\dualA} \le C^{1/2}
       \left( \sum_{k=1}^{\K+1} \| \sqrt{p_k} v\|_{\dualAk}^2 
       \right)^{1/2}.
     \]
\end{lemma}

\begin{proof}
    Combining~\cref{eq:dualA_equivalent_def} and~\cref{eq:5.1}, we have
   \begin{align*}
       \|v\|_{\dualA} & 
       \le C^{1/2} \sup_{w\in \Sobolev{1}(\mathbb R^d)} 
       \frac{\langle v,w\rangle_{\mathbb R^d}}{\left( 
       \sum_{k=1}^{\K+1} \| \sqrt{p_k} w\|_{\normAk}^2 \right)^{1/2}}.
   \end{align*}
   Moreover, since the support of $p_k$ is $\Omega_k$ for every $k$, we may
    restrict the inner products on the subdomains and use that $A_k^{-1/2}A_k^{1/2}v = v$ for any $v \in \Sobolev{1}_0(\Omega_k)$ for all $k=1,\dots,\K$ and that $A_{\K+1}^{-1/2}A_{\K+1}^{1/2} v = v$ for any $v \in \Sobolev{1}(\mathbb R^d)$.
    This justifies the following calculation 
   \begin{align*}
       \langle v,w\rangle_{\mathbb R^d} &= \sum_{k=1}^{\K+1} \langle  \sqrt{p_k} v,
        \sqrt{p_k} w\rangle_{\mathbb R^d} = \sum_{k=1}^{\K} \langle \sqrt{p_k} v,
       \sqrt{p_k} w\rangle_{\Omega_k} +  \langle \sqrt{p_{\K+1}} v,
       \sqrt{p_{\K+1}} w\rangle_{\mathbb{R}^d}\\
       &= \sum_{k=1}^{\K} \langle A_k^{-1/2} \sqrt{p_k} v,
       A_k^{1/2} \sqrt{p_k} w\rangle_{\Omega_k} +  \langle A_{\K+1}^{-1/2} \sqrt{p_{\K+1}} v,
       A_{\K+1}^{1/2} \sqrt{p_{\K+1}} w\rangle_{\mathbb{R}^d}\\
       & \le \left( \sum_{k=1}^{\K+1} \| \sqrt{p_k} v\|_{\dualAk}^2 \right)
       ^{1/2} \left( \sum_{k=1}^{\K+1} \| \sqrt{p_k} w\|_{\normAk}^2 \right)^{1/2},
   \end{align*}
   hence the result.
\end{proof}     

\subsection{Source problem}

We now turn to the proof of the error estimate for the source problem.

\begin{proof}[Proof of Theorem~\ref{thm:a_posteriori1}]
    Since $u$ satisfies $Au=f$, there holds $\Res(u_N) = A(u-u_N)$. 
    Decomposing $u-u_N$ using the partition of unity and since $u-u_N \in \Sobolev{2}(\mathbb R^d)$, 
   \[
       \norm{u-u_N}_{\normA}^2 = 
       \langle u-u_N,A(u-u_N)\rangle_{\mathbb R^d} = \sum_{k=1}^{\K+1}
       \langle \sqrt{p_k}(u-u_N), \sqrt{p_k}\Res(u_N)\rangle_{\mathbb R^d}.
   \]
    Since the support of $p_k$ is $\Omega_k$ for every $k$, we may
    restrict the inner products on the subdomains and use that $A_k^{-1/2}A_k^{1/2}v = v$ for any $v \in \Sobolev{1}_0(\Omega_k)$ for all $k=1,\dots,\K$ and that $A_{\K+1}^{-1/2}A_{\K+1}^{1/2} v = v$ for any $v \in \Sobolev{1}(\mathbb R^d)$.
    This justifies the following calculation
    \begin{align*}
        \norm{u-u_N}_{\normA}^2 &= 
       \sum_{k=1}^{\K} \langle \sqrt{p_k}(u-u_N), 
       \sqrt{p_k}\Res(u_N)\rangle_{\Omega_k}\\
       & \qquad + \langle \sqrt{p_{\K+1}}(u-u_N), 
       \sqrt{p_{\K+1}}\Res(u_N)\rangle_{\mathbb{R}^d}\\
        &= \sum_{k=1}^{\K} \langle A_k^{1/2}\sqrt{p_k}(u-u_N), 
       A_k^{-1/2} \sqrt{p_k}\Res(u_N)\rangle_{\Omega_k}\\
       & \qquad + \langle A_{\K+1}^{1/2} \sqrt{p_{\K+1}}(u-u_N), 
       A_{\K+1}^{-1/2}\sqrt{p_{\K+1}}\Res(u_N)\rangle_{\mathbb{R}^d}.\\
    \end{align*}
   By the Cauchy--Schwarz inequality on $\Omega_k$ applied to the 
    above right-hand side, one obtains the bound
   \begin{align*}
       \norm{u-u_N}_{\normA}^2 &\leq \sum_{k=1}^{\K}  
       \| \sqrt{p_k}(u-u_N)\|_{\normAk}
       \|\sqrt{p_k}\Res(u_N)\|_{\dualAk} \\
       &\qquad + \| \sqrt{p_{\K+1}}(u-u_N)\|_{\normAMp}
       \|\sqrt{p_{\K+1}}\Res(u_N)\|_{\dualAMp},
   \end{align*}
    that can be further bounded, using the discrete Cauchy--Schwarz inequality,
    as
   \begin{align*}
       \norm{u-u_N}_{\normA}^2 \leq
       \left[
           \sum_{k=1}^{\K+1} \| \sqrt{p_k} (u-u_N) \|_{\normAk}^2
        \right]^{1/2}
        \left[
            \sum_{k=1}^{\K+1} \| \sqrt{p_k}\Res (u_N) \|_{\dualAk}^2
       \right]^{1/2}.
   \end{align*}
   Evoking Lemma~\ref{lemma:lem1} yields the result.
 \end{proof}

\subsection{Eigenvalue problem}

We now present the proof of the error estimates for the eigenvalue problem, starting with the following proposition.

\begin{proposition}[$A$-error and $L^2$-error estimations]\label{prop:prop5}
For $1\leq i\leq N$, let 
    $(\lambda_{iN},\eigvec_{iN}) \in \mathbb{R} \times \Sobolev{2}(\mathbb{R}^d)$ be a solution to~\cref{eq:pb_eig_Gal}, and 
    $(\lambda_i,\eigvec_i)$ be a solution to equation \cref{eq:pb_eig}.
      Under \cref{as:gap_condition}, there holds
    \begin{align}
        \| \eigvec_i - \eigvec_{iN} \|_{\normA}^2  - \frac{\lambda_i}{4}
       \| \eigvec_i - \eigvec_{iN} \|_{\mathbb R^d}^4 
      &\le \tildeprefactor_i^{-1} \| \Res(\lambda_{iN},\eigvec_{iN}) \|_{\dualA}^2, 
      \label{eq:Aestim}\\
             \| \eigvec_i -\eigvec_{iN} \|_{\mathbb R^d}^2  - \frac{1}{4}\| \eigvec_i -
        \eigvec_{iN}\|_{\mathbb R^d}^4 
        &\le \hatprefactor_i^{-1} \|
        \Res(\lambda_{iN},\eigvec_{iN}) \|_{\dualA}^2, \label{eq:L2estim}
    \end{align}
    where $\tildeprefactor_i$ and $\hatprefactor_i$ are the positive constants defined in \cref{eq:tildeCi} and \cref{eq:tildeCihat}.
\end{proposition}

\begin{proof} 
   By definition of the residual, one has the equality
    \[
       \|  \Res(\lambda_{iN},\eigvec_{iN})  \|_{\dualA}^2 
        = \langle A^{-1/2}(A- \lambda_{iN})\eigvec_{iN} , A^{-1/2} (A- \lambda_{iN})\eigvec_{iN}
       \rangle_{\mathbb R^d}.
   \]
    First, we write
    \begin{align*}
        \|  \Res(\lambda_{iN},\eigvec_{iN})  \|_{\dualA}^2 = &
        | \langle \eigvec_i , A^{-1/2} (A- \lambda_{iN})\eigvec_{iN}
       \rangle_{\mathbb R^d}|^2
       + \| \Pi^\perp_{\eigvec_i} A^{-1/2} (A- \lambda_{iN})\eigvec_{iN}
       \|_{\mathbb R^d}^2, 
    \end{align*}
    where $\Pi^\perp_{\eigvec_i}$ is the orthogonal projector onto the orthogonal space of $\eigvec_i$ with respect to the $\Ldeux(\R^d)$ scalar product.
    Noting that $\Pi^\perp_{\eigvec_i}$
    and $A^{-1/2} (A- \lambda_{iN})$ commute,
     \begin{align*}
        \|  \Res(\lambda_{iN},\eigvec_{iN})  \|_{\dualA}^2 = \; &
        | \langle \eigvec_i , A^{-1/2} (A- \lambda_{iN})\eigvec_{iN}
       \rangle_{\mathbb R^d}|^2 \\
       & \;  + \| \Pi^\perp_{\eigvec_i} A^{-1/2} (A- \lambda_{iN})(\eigvec_{iN} - \eigvec_i)
       \|_{\mathbb R^d}^2 \\
       \ge  \; & \| \Pi^\perp_{\eigvec_i} A^{-1/2} (A- \lambda_{iN})(\eigvec_{iN} - \eigvec_i)
       \|_{\mathbb R^d}^2.
    \end{align*}
    By spectral calculus
    denoting by $d\mu_{A,\eigvec_{iN}-\eigvec_i}$ the spectral measure associated to $A$ and the function $\eigvec_{iN}-\eigvec_i$ (see~\cite[Section 4.3]{Lewin_2022}), we obtain
    \begin{align*}
        \|  \Res(\lambda_{iN},\eigvec_{iN})  \|_{\dualA}^2
        &\ge  
        \int_{\spectrum(A)\backslash \{\lambda_i\}} \frac{(s-\lambda_{iN})^2}{s} d\mu_{A,\eigvec_{iN}-\eigvec_i}(s).
         \nonumber 
    \end{align*}
Now using~\cref{eq:tildeCi} there holds
\begin{equation}
      \|  \Res(\lambda_{iN},\eigvec_{iN})  \|_{\dualA}^2 \ge 
     \tildeprefactor_i 
         \int_{\spectrum(A)\backslash \{\lambda_i\}}  s \;d\mu_{A,\eigvec_{iN} - \eigvec_i}(s) ,
        \label{eq:eig7}
\end{equation}
and 
\begin{equation}
         \|  \Res(\lambda_{iN},\eigvec_{iN})  \|_{\dualA}^2 \ge 
     \hatprefactor_i 
        \int_{\spectrum(A)\backslash \{\lambda_i\}} \;d\mu_{A,\eigvec_{iN} - \eigvec_i}(s) .
        \label{eq:eig7hat}
\end{equation}
   Moreover, employing the scalings $\|\eigvec_i\|_{\mathbb R^d}=\|\eigvec_{iN}\|_{\mathbb R^d}=1$,
    we write 
    \begin{equation}\label{eq:eig5}
        \langle \eigvec_i-\eigvec_{iN},\eigvec_i\rangle_{\mathbb R^d} = 
        \frac{\|\eigvec_i\|_{\mathbb R^d}+\|\eigvec_{iN}\|_{\mathbb R^d}}{2} - 
        \langle \eigvec_{iN},\eigvec_i\rangle_{\mathbb R^d} = 
        \frac 12\|\eigvec_i - \eigvec_{iN}\|_{\mathbb R^d}^2.
    \end{equation}
    Hence 
    \begin{align}
             \int_{\spectrum(A)\backslash \{\lambda_i\}}  s \;d\mu_{A,\eigvec_{iN} - \eigvec_i}(s)
            &=
            \| \eigvec_{iN} - \eigvec_i\|_{A}^2 
            - \lambda_i \left|\langle \eigvec_{iN} - \eigvec_i,
            \eigvec_i\rangle_{\mathbb R^d}\right|^2 \nonumber
            \\
            &
            =  
            \| \eigvec_{iN} - \eigvec_i\|_{A}^2 - \frac{\lambda_i}{4}
            \| \eigvec_{iN} - \eigvec_i\|_{\mathbb R^d}^4. \label{eq:ssss}
    \end{align}
    Also, 
       \begin{align}
            \int_{\spectrum(A)\backslash \{\lambda_i\}}  \;d\mu_{A,\eigvec_{iN} - \eigvec_i}(s) 
            &=
            \| \eigvec_{iN} - \eigvec_i\|_{\mathbb R^d}^2 - \left|\langle \eigvec_{iN} - \eigvec_i,
            \eigvec_i\rangle_{\mathbb R^d}\right|^2 \nonumber \\
            &=  \| \eigvec_{iN} - \eigvec_i\|_{\mathbb R^d}^2 - \frac 14 
            \| \eigvec_{iN} - \eigvec_i\|_{\mathbb R^d}^4. \label{eq:tttt}
    \end{align} 
    Combining~\cref{eq:eig7} and~\cref{eq:ssss}, we obtain~\cref{eq:Aestim}.
    Similarly, combining~\cref{eq:eig7hat} and~\cref{eq:tttt}, we obtain~\cref{eq:L2estim}.
    \end{proof}

\begin{lemma}\label{lemma:prop4}
    With the same notation as Proposition~\ref{prop:prop5}, 
    under \cref{as:gap_condition}
    and under the assumption $\langle \eigvec_i,\eigvec_{iN}\rangle_{\mathbb R^d}\geq 0$, there holds
    \[
        \|\eigvec_i - \eigvec_{iN}\|_{\mathbb R^d}\leq (2\hatprefactor_i^{-1})^{1/2}
        \|\Res(\lambda_{iN},\eigvec_{iN})\|_{\dualA}.
    \]
\end{lemma}

\begin{proof} 
    Using the assumption $\langle \eigvec_i,\eigvec_{iN}\rangle_{\mathbb R^d}\geq 0$
    and employing the scaling 
    $\|\eigvec_i\|_{\mathbb R^d}=\|\eigvec_{iN}\|_{\mathbb R^d} = 1$, one
    obtains the inequality
    \[
        \|\eigvec_i - \eigvec_{iN}\|_{\mathbb R^d}^2 = 2 - 2\langle \eigvec_i,\eigvec_{iN}\rangle_{\mathbb
        R^d} \leq 2.
    \]
    Writing
    \begin{align*}
        \|\eigvec_i - \eigvec_{iN}\|_{\mathbb R^d}^2 - &\frac 14\|\eigvec_i - \eigvec_{iN}\|_{\mathbb
        R^d}^4
        =
         \|\eigvec_i - \eigvec_{iN}\|_{\mathbb R^d}^2 - \frac 14\|\eigvec_i - \eigvec_{iN}\|_{\mathbb
        R^d}^2\|\eigvec_i - \eigvec_{iN}\|_{\mathbb R^d}^2 \\
        &\geq   
        \|\eigvec_i - \eigvec_{iN}\|_{\mathbb R^d}^2 - \frac 12\|\eigvec_i - \eigvec_{iN}\|_{\mathbb
        R^d}^2 = 
        \frac 12 \|\eigvec_i - \eigvec_{iN}\|_{\mathbb R^d}^2,
    \end{align*}
    and evoking~\cref{eq:L2estim} yields the result.
\end{proof}

\begin{proposition}[Eigenvalue estimation]
\label{prop:eigenv_prop57}
    With the same notation as Proposition~\ref{prop:prop5}, 
    under \cref{as:gap_condition},
    we have
   \[
     \lambda_{iN} - \lambda_i \le
     \tildeprefactor_i^{-1} \| \Res(\lambda_{iN},\eigvec_{iN}) \|_{\dualA}^2.
   \]
\end{proposition}

\begin{proof}
   Evoking equation \cref{eq:eig5}, there holds
    \begin{align*} 
        \|\eigvec_i-\eigvec_{iN}\|_{\normA}^2 = 
        \lambda_i +\lambda_{iN} - 2 \lambda_i \langle \eigvec_i,\eigvec_{iN}\rangle_{\mathbb
        R^d} \nonumber &= \lambda_{iN} - \lambda_i + 2\lambda_i (1 -  \langle
        \eigvec_i,\eigvec_{iN}\rangle_{\mathbb R^d}) \nonumber \\
        & = \lambda_{iN} - \lambda_i + \lambda_i \| \eigvec_i-\eigvec_{iN}\|_{\mathbb R^d}^2.
   \end{align*}
   Hence $
      \lambda_{iN} - \lambda_i =  \|\eigvec_i-\eigvec_{iN}\|_A^2 - \lambda_i \| \eigvec_i-\eigvec_{iN}\|_{\mathbb R^d}^2.$
Since \(\frac{\| \eigvec_i-\eigvec_{iN}\|_{\mathbb R^d}^2}{4} \le 1, \)
there holds
\[
\lambda_{iN} - \lambda_i \le  \|\eigvec-\eigvec_{iN}\|_A^2 - \frac{\lambda_i}{4} \| \eigvec_i-\eigvec_{iN}\|_{\mathbb R^d}^4.
\]
Using~\cref{eq:Aestim} yields the bound.
\end{proof}

\begin{proof}[Proof of Theorem~\ref{thm:a_posteriori2}]
    
    Employing Lemma~\ref{lemma:prop4} within Proposition~\ref{prop:prop5} yields
    \begin{align*}
       \| \eigvec_i - \eigvec_{iN} \|_{\normA}^2
        &\le \tildeprefactor_i^{-1} \| \Res(\lambda_{iN},\eigvec_{iN}) \|_{\dualA}^2 + \frac{\lambda_i}{4}
       \| \eigvec_i - \eigvec_{iN} \|_{\mathbb R^d}^4 \\
        &\le \tildeprefactor_i^{-1} \| \Res(\lambda_{iN},\eigvec_{iN}) \|_{\dualA}^2 + \lambda_i\hatprefactor_i^{-2}  
        \| \Res(\lambda_{iN},\eigvec_{iN}) \|_{\dualA}^4. 
    \end{align*}
    Evoking Lemma~\ref{lemma:lem2} yields~\cref{eq:Aestimthm}.
    For the eigenvalue estimate \cref{eq:eigenestimthm}, by combining Proposition~\ref{prop:eigenv_prop57} and Lemma~\ref{lemma:lem2}, we have 
    \[
        \lambda_{iN} - \lambda_i \leq \tildeprefactor_i^{-1} \| \Res(\lambda_{iN},\eigvec_{iN})\|_{\dualA}^2 \leq C \tildeprefactor_i^{-1}\sum_{k=1}^{\K+1} \| \sqrt{p_k} \Res(\lambda_{iN},\eigvec_{iN})\|_{\dualAk}^2.
    \]
    The lower bound is a simple consequence of the variational principle.
  \end{proof}

 \begin{proof}[Proof of Lemma~\ref{lemma:uplow}]
    Let $v\in \Ldeux(\Omega_k)$.
    Using that the eigenvalues $(\varepsilon_{j}^{(k)})_j$ are increasing, we have 
     \[
         \|v\|_{\dualAk}^2 = \mathcal I_k(v) + \sum_{j=\J_k+1}^\infty
         \frac{1}{\varepsilon_j^{(k)}}\left|\langle
         v,\psi_j^{(k)}\rangle_{\Omega_k}\right|^2 \leq  \mathcal I_k(v) + 
         \frac{1}{\varepsilon_{\J_k}^{(k)}}
        \sum_{j=\J_k+1}^\infty \left|\langle
         v,\psi_j^{(k)}\rangle_{\Omega_k}\right|^2.
     \]
    Using that 
     \[
         \|v\|_{\Omega_k}^2 = \sum_{j=1}^{\J_k} \left|\langle
         v,\psi_j^{(k)}\rangle_{\Omega_k}\right|^2 + 
         \sum_{j=\J_k+1}^\infty \left|\langle
         v,\psi_j^{(k)}\rangle_{\Omega_k}\right|^2,
     \]
     we obtain the upper bound.
     The lower bound is straightforward.
\end{proof}

 \begin{proof}[Proof of Theorem~\ref{thm:practical1}]
     For $k=1,\ldots,K$, let $w_k=\sqrt{p_k}\Res(u_N) \in \Ldeux(\Omega_k)$. 
     Applying \cref{lemma:uplow} to $w_k$  
     yields the upper bound
    \[
        \|w_k\|^2_{\dualAk}
        \leq
        \mathcal I_k(w_k)
        + \frac{1}{\varepsilon^{(k)}_{\J_k}}\left(\|w_k\|_{\Omega_k}^2 -
        \sum_{j=1}^{\J_k} 
        \left|\langle w_k,\psi^{(k)}_j\rangle_{\Omega_k}\right|^2\right).
    \]
    Using the bound in Theorem~\ref{thm:a_posteriori1} finishes the proof.
 \end{proof}
 
 \begin{proof}[Proof of Theorem~\ref{thm:practical2}]
     The proof is similar to the proof of \cref{thm:practical1} using \cref{thm:a_posteriori2} instead of \cref{thm:a_posteriori1}.
 \end{proof}

\section{Conclusion}
\label{sec:concl}

In this work, we have proposed {\it a posteriori} error bounds for source problems and eigenvalue problems, that is suited to linear combinations of atomic orbitals basis sets, which are the standard discretization choice in {\it ab initio} electronic molecular calculations. 
These error bounds exploit the spherical symmetry of the potentials, locally around the nuclei.
Moreover, since these {\it a posteriori} bounds are localized around each nucleus of the considered system, it gives a criterion to dynamically adapt the basis set. 
This is particularly important as there is no {\it a priori} error estimates for the Gaussian basis sets used in practice.
The extension of this contribution to nonlinear models (Hartree--Fock, Kohn--Sham models) in the spirit of~\cite{Bordignon2024-fp}, as well as error bounds for electronic properties, \emph{e.g.} for derivatives of the eigenvalues with respect to parameters appearing in the potential, as in~\cite{Cances2022-mg}, is left for future work.

\section*{Acknowlegments}

The authors are grateful to Susi Lehtola for contributing to the implementation of the code on three-dimensional models. The authors would like to thank Yvon Maday for fruitful discussions and comments.

\bibliographystyle{siamplain}
\bibliography{references}

\end{document}